\documentclass[A4]{amsart}
\usepackage[square,sort,comma,numbers]{natbib}
\usepackage{amsmath, amssymb, amstext, amsfonts, textcomp, amsxtra, amsbsy, amsgen, amsopn, amscd, mathrsfs, amsthm, latexsym, array}
\usepackage[textwidth=16cm,textheight=22cm,centering]{geometry}

\usepackage{color}
\usepackage{graphicx}

\usepackage[all]{xy}

\newtheorem{theorem}             {Theorem}  [section]
\newtheorem{definition} [theorem] {Definition}
\newtheorem{lemma}      [theorem]{Lemma}
\newtheorem{corollary}  [theorem]{Corollary}
\newtheorem{proposition}[theorem]{Proposition}
\newtheorem{remark} [theorem] {Remark}

\numberwithin{equation}{section} \everymath{\displaystyle}

\newcommand{\Cont}{{\rm C}}
\newcommand{\Aut}{\mathcal{A}}

\newcommand{\Sch}{\mathcal{S}}

\newcommand{\intL}{{\rm L}}
\newcommand{\Ht}{{\rm Ht}}

\newcommand{\Nr}{{\rm Nr}}

\newcommand{\BesselK}{\mathcal{K}}
\newcommand{\hol}{{\rm hol}}

\newcommand{\gp}[1]{\mathbf{#1}}
\newcommand{\GL}{{\rm GL}}
\newcommand{\PGL}{{\rm PGL}}

\newcommand{\PSL}{{\rm PSL}}
\newcommand{\SO}{{\rm SO}}
\newcommand{\PSO}{{\rm PSO}}
\newcommand{\SU}{{\rm SU}}
\newcommand{\grG}{\mathbf{G}}

\newcommand{\grK}{\mathbf{K}}
\newcommand{\grZ}{\mathbf{Z}}

\newcommand{\grB}{\mathbf{B}}

\newcommand{\grN}{\mathbf{N}}

\newcommand{\ag}[1]{\mathbb{#1}}
\newcommand{\N}{\mathbb{N}}
\newcommand{\Z}{\mathbb{Z}}

\newcommand{\Q}{\mathbb{Q}}
\newcommand{\R}{\mathbb{R}}
\newcommand{\C}{\mathbb{C}}
\newcommand{\F}{\mathbf{F}}
\newcommand{\A}{\mathbb{A}}
\newcommand{\vo}{\mathfrak{o}}
\newcommand{\vp}{\mathfrak{p}}
\newcommand{\idlA}{\mathfrak{A}}

\newcommand{\ProjP}{{\rm P}}
\newcommand{\norm}[1][\cdot]{\lvert #1 \rvert}
\newcommand{\extnorm}[1]{\left\lvert #1 \right\rvert}
\newcommand{\Norm}[1][\cdot]{\lVert #1 \rVert}
\newcommand{\extNorm}[1]{\left\lVert #1 \right\rVert}

\newcommand{\Pairing}[2]{\langle #1, #2 \rangle}

\newcommand{\Four}[2][]{\mathfrak{F}_{#1}(#2)}

\newcommand{\Mellin}[2][]{\mathfrak{M}_{#1}(#2)}

\newcommand{\rpR}{{\rm R}}

\newcommand{\Res}{{\rm Res}}
\newcommand{\Ind}{{\rm Ind}}

\newcommand{\Intw}{\mathcal{M}}
\newcommand{\IntwR}{\mathcal{R}}
\newcommand{\Whi}{\mathcal{W}}
\newcommand{\Kir}{\mathcal{K}}
\newcommand{\Cond}{\mathbf{C}}

\newcommand{\fin}{{\rm fin}}
\newcommand{\eis}{{\rm E}}
\newcommand{\Reis}{\mathcal{E}}
\newcommand{\eisCst}{{\rm E}_{\grN}}
\newcommand{\reg}{{\rm reg}}

\newcommand{\freg}{{\rm fr}}

\newcommand{\pH}{\mathbb{H}}

\newcommand{\Vol}{{\rm Vol}}
\makeatletter

\newcommand{\Rmnum}[1]{\expandafter\@slowromancap\romannumeral #1@}
\makeatother

\newcommand{\Ex}{\mathcal{E}{\rm x}}

\title{Burgess-like Subconvexity for $\GL_1$}
\author{Han Wu}
\thanks{Research partially supported by SNF-grant 200021-125291 and DFG-SNF-grant 00021L\_153647 \\ 2010 Mathematics Subject Classification [MSC] codes: 11R42 \\ Keywords: subconvexity, $L$-function over number field, Burgess-type hybrid bound}

\begin{document}
	
	\begin{abstract}
		We generalize our previous method on subconvexity problem for $\GL_2 \times \GL_1$ with cuspidal representations to Eisenstein series, and deduce a Burgess-like subconvex bound for Hecke characters, i.e., the bound $\norm[L(1/2,\chi)] \ll_{\F,\epsilon} \Cond(\chi)^{1/4-(1-2\theta)/16+\epsilon}$ for varying Hecke characters $\chi$ over a number field $\F$ with analytic conductor $\Cond(\chi)$. As a main tool, we apply the extended theory of regularized integral due to Zagier developed in a previous paper to obtain the relevant triple product formulas of Eisenstein series.
	\end{abstract}
	
	\maketitle
	
	\tableofcontents

\section{Introduction}

	\subsection{Statement of Main Result}
	
	If $\pi$ is a (cuspidal) automorphic representation of $\GL_d$ over a number field $\F$ with (usual) conductor $\Cond(\pi_{\fin})$ resp. archimedean analytic conductor $\Cond(\pi_{\infty})$ resp. analytic conductor $\Cond(\pi) = \Cond(\pi_{\infty})\Cond(\pi_{\fin})$, the absolute convergence for $\Re s > 1$ of the associated $L$-function $L(s,\pi)$ and the functional equation implies for any $\epsilon > 0$, via the Phragm\'en-Lindel\"of principle together with Iwaniec's method \cite[(19.16)]{DFI02} \footnote{The cited argument only treats the case for $\F=\Q$ but it works for general number fields by replacing the relevant divisor function by the one for ideals of the ring of algebraic integers.} to establish the necessary bound on the vertical line with $\Re s = 1+\epsilon$ with small $\epsilon > 0$, the estimation at the central point
	$$ \extnorm{L(1/2,\pi)} \ll_{\F, \epsilon} \Cond(\pi)^{1/4+\epsilon}, $$
called the \emph{convex bound} or the \emph{convexity}. If the Riemann Hypothesis holds for $L(s,\pi)$, then we have the optimal bound
	$$ \extnorm{L(1/2,\pi)} \ll_{\F, \epsilon} \Cond(\pi)^{\epsilon}, $$
called the \emph{Lindel\"of Hypothesis}. Reducing the exponent of $\Cond(\pi)$ from $1/4+\epsilon$ to $1/4-\delta+\epsilon$ for some positive constant $0< \delta < 1/4$ is called the \emph{subconvexity problem}. More generally, for $Q = \Cond(\pi)$ resp. $\Cond(\pi_{\fin})$ resp. $\Cond(\pi_{\infty})$, an estimation
	$$ \extnorm{L(1/2,\pi)} \ll_{\F, \epsilon, \Cond(\pi)/Q} Q^{1/4-\delta+\epsilon} $$
is called a \emph{(hybrid) subconvex bound} resp. \emph{subconvex bound in the level aspect} resp. \emph{subconvex bound in the archimedean aspect}.

	In the simplest case, the first and most famous subconvex bound was obtained for the Riemann zeta-function by Weyl \cite{Wyl21} (see for example \cite[\S 6.6]{Pa88})
	$$ \zeta(1/2+it) \ll_{\epsilon} \norm[t]^{1/4-1/12+\epsilon}, \quad t \in \ag{R}, $$
which can be considered as (a special case of) a subconvex bound in the archimedean aspect for the Dirichlet $L$-functions. If $\chi$ is a Dirichlet character of modulus $q=\Cond(\chi_{\fin}) \in \ag{N}$, Burgess \cite{Bu63} established his famous subconvex bound
	$$ \norm[L(1/2, \chi \norm_{\ag{A}}^{it})] = \norm[L(1/2+it,\chi)] \ll_{t,\epsilon} q^{\frac{1}{4}-\frac{1}{16}+\epsilon}. $$
	Later, Heath-Brown \citep{HB78,HB80} generalized Burgess' result to include the $t$-aspect as the following hybrid bound
	$$ \norm[L(1/2+it,\chi)] \ll_{\epsilon} (q(\norm[t]+2))^{\frac{1}{4}-\frac{1}{16}+\epsilon}. $$
	Ever since, the subconvexity problem has become a venerable problem in analytic number theory, in which both the optimal subconvex saving $\delta$ and the largest class of $L$-function mark the limit of techniques of analytic number theory. The saving $\delta = 1/12$ resp. $1/16$ seem to be two natural barriers in the literature. They are called \emph{Weyl-type} resp. \emph{Burgess-type} subconvex bound for this historic reason. Moreover, it was discovered that for $d > 1$ the subconvexity problem of $L(s,\pi)$ is intimately related with various equidistribution problems \cite{Du88, Sar01}. More such relations can be found in \cite[Lecture 5]{Mi07}, as well as an application of the subconvexity problem in the level aspect for $d=1$ and $\F$ imaginary quadratic.
	
	In this paper, we restrict to the case $d=1$, i.e., when $\pi=\chi$ is a Hecke character. In the case $\F=\ag{Q}$, many strong results are known besides the above bounds by Weyl, Burgess and Heath-Brown (for example \cite{HW00}), especially in some special cases. For example, in \cite{HW00} the case of $q$ prime and of hybrid type is considered; in \cite{Mil16} with very strong result of sub-Weyl type, the case of $q=p^n$ a prime power and for the $q$-aspect is treated. Another interesting special case is when we restrict to $\chi=\chi_q$ the quadratic character (and for $q$ \emph{special}, say \emph{square-free}). Bounds of better quality than Burgess' are known to hold for Weyl-type. For example, among many other good results Conrey and Iwaniec \cite[Corollary 1.5]{CI00} obtained
	$$ \norm[L(1/2+it,\chi_q)] \ll_{t,\epsilon} q^{\frac{1}{4}-\frac{1}{12}+\epsilon}, $$
	which was recently generalized by Young \cite[(1.5)]{Yo17} as
	$$ \norm[L(1/2+it,\chi_q)] \ll_{\epsilon} (q(\norm[t]+2))^{\frac{1}{4}-\frac{1}{12}+\epsilon}. $$
	The above bound was further generalized for \emph{cube-free} $q$ by Petrow and Young \cite{PY18}. Over a general number field, the best known result is the main theorem of Soehne \cite[p.227]{Soe97}, which follows the method of Heath-Brown \cite{HB78, HB80} (it attains the Weyl-type bound if the usual conductor $\mathfrak{f} = \mathfrak{f}_0^3$ is a \emph{cube}):
	$$ L(1/2+it, \chi) \ll_{\epsilon, \F} \left( \Cond_{\infty}(\chi \norm_{\ag{A}}^{it}) \Nr(\mathfrak{f}) \right)^{1/6+\epsilon} + \Nr(\mathfrak{f}_0)^{1/2+\epsilon} + \Nr(\mathfrak{f} / \mathfrak{f}_0)^{1/4+\epsilon}, $$
	where $\mathfrak{f}_0$ is any ideal dividing $\mathfrak{f}$, the usual conductor of $\chi$.
	
	 In the work of Michel \& Venkatesh \cite[Theorem 5.1 \& Section 5.1.7]{MV10}, a subconvex bound for Hecke characters $\chi$ was obtained with the subconvex exponent unspecified. We shall modify their approach and obtain a hybrid subconvex bound of Burgess-type for $L$-functions associated with Hecke characters over general number fields.
\begin{theorem}
	Let $\chi$ be a Hecke character of $\F$ with analytic conductor $\Cond(\chi)$. We have
	$$ \extnorm{L(\frac{1}{2},\chi)} \ll_{\F,\epsilon} \Cond(\chi)^{\frac{1}{4}-\frac{1-2\theta}{16}+\epsilon}, $$
	where $\theta$ is any constant towards the Ramanujan-Petersson conjecture.
\label{MainThm}
\end{theorem}

\noindent Combining with Soehne's bound, we deduce the following result.

\begin{corollary}
	Let $\chi$ be a Hecke character of $\F$ and $t \in \ag{R}$. Denote by
	$$ T := \Cond_{\infty}(\chi \cdot \norm_{\ag{A}}^{it}), \quad q := \Cond_{\fin}(\chi). $$
	Then we have for any $\epsilon > 0$
	$$ \extnorm{ L(\frac{1}{2}+it, \chi) } \ll_{\epsilon, \F} (Tq)^{\epsilon} \cdot \left\{ \begin{matrix} (Tq)^{1/6} & \text{if } T \geq q^{1/2} \\ q^{1/4} & \text{if } q^{(1-2\theta)/(3+2\theta)} \leq T < q^{1/2} \\ (Tq)^{(3+2\theta)/16} & \text{if } T \leq q^{(1-2\theta)/(3+2\theta)} \end{matrix} \right. , $$
	where $\theta$ is any constant towards the Ramanujan-Petersson conjecture.
\end{corollary}
\begin{proof}
	We apply Soehne's bound with $\mathfrak{f}_0 = \vo$ the ring of integers of $\F$, compare it with Theorem \ref{MainThm} and distinguish cases according to the relative size of $T$ and $q$.
\end{proof}

	\subsection{Discussion on Method}
	
	The proof is inspired by the method of our earlier work \cite{Wu14} where we established a Burgess-type subconvext bound for $\GL_1$ twists of a $\GL_2$ {\em cuspidal} representation $\pi$.
\begin{equation}
	\extnorm{L(1/2,\pi \times \chi)} \ll_{\pi,\F,\epsilon} \Cond(\chi)^{\frac{1}{2}-\frac{1-2\theta}{8}+\epsilon}.
\label{MainBd}
\end{equation}
	In this paper we show that it is possible to replace the cuspidal representation $\pi$ by the Eisenstein series representation $\pi(1,1)$ and obtain the same bound. Theorem \ref{MainThm} then follows from the identity
\begin{equation}
	L(s, \pi(1,1) \times \chi) = L(s,\chi)^2.
\label{GL2GL1LRel}
\end{equation}
	The main hurdle is to address the non square-integrability of Eisenstein series. For this we use a regularization process which we show does not harm the quality of the final out-come. By contrast, the original approach with truncation on Eisenstein series \cite[\S 5.1.7]{MV10} does destroy the Burgess-like quality (see \S \ref{FailTE} below for more details).
	
	It is worthwhile to give some comments on our method, which is quite different from the methods applied in the case $\F=\ag{Q}$ by Burgess or Conrey-Iwaniec. Burgess' method is based on the study of character sums of the shape
	$$ \sum_{m_1 \leq n \leq m_2} \chi(n), $$
which makes use of Weil's bound hence makes extensive use of the periodicity of the summand function $n \mapsto \chi(n)$. Its direct generalization
\begin{equation}
	\sum_{m_1 \leq \Nr(\idlA) \leq m_2} \chi(\idlA),
\label{BurgessChSNF}
\end{equation}
where $\idlA$ runs over integral ideals, loses the periodicity for the summand function. Our method can be viewed as a variant of Conrey-Iwaniec's method (see \cite[\S 1.1]{Wu3}). The main common feature is to bring a problem for $\GL_1$ into the setting for $\GL_2$, and to use the available knowledge on the spectral theory of automorphic representations for $\GL_2$. According to the comparison between \cite{BH10} and \cite{Wu14}, this method virtually consists of taking (\ref{BurgessChSNF}) into the fourth power and studying the cancellation of the resulted character sums by means of the ``spectral decomposition of the shifted convolution sums'' \cite{BH08_Sh} instead of Weil's bounds on character sums, hence it is also a variant of the original method of Burgess, ``disguised'' into the language of periods, which treats the archimedean aspect equally well.

	It would also be enlightening to point out the following explanation of the identity (\ref{GL2GL1LRel}) in terms of an identity of periods. Any function in the induced model of $\pi(1,1)$ can be constructed from a Schwartz function $\Phi \in \Sch(\ag{A}^2)$ as
	$$ f_{\Phi}(g) := \norm[\det g]_{\ag{A}}^{\frac{1}{2}+s} \int_{\ag{A}^{\times}} \Phi((0,t)g) \norm[t]_{\ag{A}}^{1+2s} d^{\times}t \mid_{s=0}, $$
whose Whittaker function is equal to
	$$ W_{\Phi}(a(y)) = \norm[y]_{\A}^{\frac{1}{2}} \int_{\ag{A}^{\times}} \Four[2]{\Phi}(t,\frac{y}{t}) d^{\times}t, $$
where $\Four[2]{\Phi}$ is the Fourier transform of $\Phi$ with respect to the second variable. Hence the period representing the left hand side of (\ref{GL2GL1LRel}) can be re-written, with the change of variables $y \mapsto yt$, as
	$$ \int_{\ag{A}^{\times}} W_{\Phi}(a(y)) \chi(y) \norm[y]_{\ag{A}}^{s-\frac{1}{2}} d^{\times}y = \int_{\ag{A}^{\times} \times \ag{A}^{\times}} \Four[2]{\Phi}(t,y) \chi(ty) \norm[t]_{\ag{A}}^s \norm[y]_{\ag{A}}^s d^{\times}t d^{\times}y. $$
	The right hand side of the above equation is exactly the integral representation of that of (\ref{GL2GL1LRel}) if $\Four[2]{\Phi}$ is decomposable as the tensor product of two functions in $\Sch(\ag{A})$. In other words, if we read the above discussion inversely, we see that our method makes use of the ``two dimensional Tate's integral'', which brings in the structure of $\GL_2$ not present in the usual integral representation \`a la Tate of $L(s,\chi)$.
\begin{remark}
	We also remark that the (global) factorizability $\chi(ty) = \chi(t) \chi(y)$ is responsable for such a link. For example, even though we have a similar identity of $L$-functions
	$$ L(s, \pi(1,1) \times \pi) = L(s,\pi)^2 $$
where $\pi$ is a cuspidal representation of $\GL_2$, an identity of the integral representations of the two sides does not seem to exist: the LHS is represented by the Rankin-Selberg integral for $\GL_2 \times \GL_2$; the RHS is represented by twice/square the integral representation for the standard $L$-function for $\GL_2$.
\end{remark}

	\subsection{Notations and Conventions}
	
	$\ag{N}$ is the set of natural numbers containing $0$. All characters including Hecke characters are unitary. Non unitary ones will be called quasi-characters.
	
	If $f$ is a meromorphic function around $s=s_0$, we introduce the coefficients into its Laurent expansion
	$$ f(s) = \sideset{}{_{-\infty < k < 0}} \sum \frac{f^{(k)}(s_0)}{(-k)!} (s-s_0)^k + \sideset{}{_{k \geq 0}} \sum \frac{f^{(k)}(s_0)}{k!} (s-s_0)^k. $$
	The terms for $k<0$ form the \emph{principal part} \cite[p.p. 211]{Ru86} of $f$ at $s_0$. We write
	$$ f^{\hol}(s) = f(s) - p(s), \quad p(s) := \sideset{}{_{-\infty < k < 0}} \sum \frac{f^{(k)}(s_0)}{(-k)!} (s-s_0)^k. $$
	In particular, if $f$ has a pole of order $k_0$ at $s_0$, we have
	$$ \frac{\partial^k f^{\hol}}{\partial s^k}(s_0) = \frac{k!}{(k+k_0)!} \left. \frac{\partial^{k+k_0}}{\partial s^{k+k_0}} \right|_{s=s_0} \left( (s-s_0)^{k_0} f(s) \right). $$
\begin{remark}
	The value $f^{\hol}(s_0)$ with notations as given above is intimately related with the \emph{finite part} functional, denoted by $\mathrm{f.p.}$ in \cite[Theorem (6.33)]{GJ79}.
\end{remark}
	
	In addition to the notations given above, we import \cite[Section 2.1]{Wu14}, in which most of the notations are in fact standard. (For example, our normalization of measures is just the Tamagawa measure with the standard convergence factors.) We simply address the following points/differences.

\noindent (1) The number field is written in bold character $\F$, with ring of algebraic integers $\vo$ and ring of adeles $\A$. $v$ denotes a place of $\F$. If $v < \infty$ is finite, we usually write $v=\vp$, which is identified with a prime ideal $\vp$ of $\vo$.

\noindent (2) We write the algebraic groups defined over $\F$ in bold characters such as $\grG, \gp{N}, \grB,\grZ$ etc, where $\grG = \GL_2$, $\grB$ is the upper triangular subgroup of $\grG$, $\gp{N} \vartriangleleft \gp{B}$ is the unipotent upper triangular subgroup, and $\grZ$ is the center of $\grG$.

\noindent (3) $\gp{K} = \sideset{}{_v} \prod \gp{K}_v$ is the standard maximal compact subgroup of $\GL_2(\A)$, i.e.
	$$ \gp{K}_v = \left\{ \begin{matrix} \SO_2(\R) & \text{if } \F_v = \R \\ \SU_2(\C) & \text{if } \F_v = \C \\ \GL_2(\vo_{\vp}) & \text{if } v=\vp < \infty \end{matrix} \right. . $$
	
\noindent (4) In $\GL_2$, for local or global variables $x \in \F_v$ or $\A$, $y \in \F_v^{\times}$ or $\A^{\times}$, we write
	$$ n(x) = \begin{pmatrix} 1 & x \\ & 1 \end{pmatrix}, \quad a(y) = \begin{pmatrix} y & \\ & 1 \end{pmatrix}. $$
	
\noindent (5) We use the abbreviation
	$$ [\GL_2] = \GL_2(\F) \gp{Z}(\ag{A}) \backslash \GL_2(\ag{A}) = [\PGL_2]. $$
	
\noindent (6) If $f_0 \in \pi(1,1)$ the global principal series representation induced from trivial characters, which defines a flat section $f_s \in \pi(\norm_{\A}^s, \norm_{\A}^{-s})$, we normalize the usual Eisenstein series $\eis(s,f_0) = \eis(f_s)$ by
	$$ \eis^*(s,f_0) := \Lambda_{\F}(1+2s) \eis(s,f_0). $$
	
\noindent (7) In the above equation, $\Lambda_{\F}(s)$ is the complete Dedekind zeta function of $\zeta_{\F}(s)$. More generally, $L(\cdot)$ denotes $L$-functions without factors at infinity. $\Lambda(\cdot)$ denotes the complete $L$-functions. We write $\zeta_{\F}^*$ for the residue of $\zeta_{\F}(s)$ at $s=1$. We also introduce
	$$ \lambda_{\F}(s) = \frac{\Lambda_{\F}(-2s)}{\Lambda_{\F}(2+2s)} = \frac{\lambda_{\F}^{(-1)}(0)}{s} + O(1). $$


	Additional notations will be given in the course of proofs.

\section{Miscellaneous Preliminaries}

	\subsection{Extension of Zagier's Regularized Integral}
	
	In this subsection we recall and summarize our extension of the theory of regularized integrals in \cite[\S 5 \& \S 6]{Wu9} without proofs. This extension fits well in the context of the \emph{Rankin-Selberg trace formula}. It could not be well understood in the framework of the subconvexity problem. Hence we encourage the interested reader to read \cite[\S 5 \& \S 6]{Wu9} for a better understanding.
	
	We begin with the recall on the following space of functions on the automorphic quotient of $\GL_2$ over a general number field $\F$ with the ring of adeles $\ag{A}$.
\begin{definition}
	(\cite[Definition 5.14]{Wu9}) Let $\omega$ be a unitary character of $\F^{\times} \backslash \ag{A}^{\times}$. Let $\varphi$ be a smooth function on $\GL_2(\F) \backslash \GL_2(\ag{A})$ with central character $\omega$. We call $\varphi$ \emph{finitely regularizable} if there exist unitary characters $\chi_i: \F^{\times} \backslash \ag{A}^{\times} \to \ag{C}^{(1)}$, $\alpha_i \in \ag{C}, n_i \in \ag{N}$ and smooth functions $f_i \in \Ind_{\gp{B}(\ag{A}) \cap \gp{K}}^{\gp{K}} (\chi_i, \omega \chi_i^{-1})$ for $1 \leq i \leq l$, such that
\begin{itemize}
	\item[(1)] for any $M \gg 1$
	$$ \varphi(n(x)a(y)k) = \varphi_{\gp{N}}^*(n(x)a(y)k) + O(\norm[y]_{\ag{A}}^{-M}), \text{ as } \norm[y]_{\ag{A}} \to \infty, $$
	\item[(2)] we can differentiate the above equality with respect to the universal enveloping algebra of the lie algebra of $\GL_2(\ag{A}_{\infty})$.
\end{itemize}
	Here we have written/defined the \emph{essential constant term}
	$$ \varphi_{\gp{N}}^*(n(x)a(y)k)=\varphi_{\gp{N}}^*(a(y)k) = \sideset{}{_{i=1}^l} \sum \chi_i(y) \norm[y]_{\ag{A}}^{\frac{1}{2}+\alpha_i} \log^{n_i} \norm[y]_{\ag{A}} \cdot f_i(k). $$
	In this case, we call $\Ex(\varphi)=\{ \chi_i \norm^{\frac{1}{2}+\alpha_i}: 1 \leq i \leq l \}$ the \emph{exponent set} of $\varphi$, and define
	$$ \Ex^+(\varphi) = \{ \chi_i \norm^{\frac{1}{2}+\alpha_i} \in \Ex(\varphi): \Re \alpha_i \geq 0 \}; \quad \Ex^-(\varphi) = \{ \chi_i \norm^{\frac{1}{2}+\alpha_i} \in \Ex(\varphi): \Re \alpha_i \leq 0 \}. $$
	The space of finitely regularizable functions with central character $\omega$ is denoted by $\Aut^{\freg}(\GL_2,\omega)$.
\label{FinRegFuncDef}
\end{definition}
\noindent Obviously $\Aut^{\freg}(\GL_2,\omega)$ is stable under the right regular translation of $\GL_2(\ag{A})$ and contains the Schwartz space with central character $\omega$, hence the space of smooth cusp forms. It also contains any finite product of Eisenstein series (\cite[Remark 5.19]{Wu9}). In the case $\omega = 1$ and for any $\varphi \in \Aut^{\freg}(\GL_2,1)$, the integral
	$$ R(s,\varphi) := \int_{\ag{A}^{\times} \times \gp{K}} (\varphi_{\gp{N}} - \varphi_{\gp{N}}^*)(a(y)\kappa) \norm[y]_{\ag{A}}^{s-1/2} d^{\times}y d\kappa $$
is convergent for any $\Re s >> 1$ and admits meromorphic continuation. We use it to define the regularized integral as
	$$ \Aut^{\freg}(\GL_2,1) \to \ag{C}, $$
	$$ \varphi \mapsto \int_{[\PGL_2]}^{\reg} \varphi(g) dg := \frac{1}{\Vol([\PGL_2])} \left( \Res_{s=1/2} R(s,\varphi) + \sideset{}{_{\substack{ \alpha_i = -1 \\ n_i = 0 \\ \chi_1(\ag{A}^{(1)}) = 1 }}} \sum \int_{\gp{K}} f_i(\kappa) d\kappa \right). $$

	If $f \in \Ind_{\gp{B}(\ag{A})}^{\GL_2(\ag{A})}(\chi_1,\chi_2)$ such that $\chi_1 \chi_2^{-1} = \norm_{\ag{A}}^{i\mu}$ for some $\mu \in \ag{R}$, we introduce the \emph{regularizing Eisenstein series} as (\cite[Definition 5.16]{Wu9})
\begin{equation}
	\eis^{\reg}(s, f)(g) = \eis(s, f)(g) - \frac{\Lambda_{\F}(1-2s-i\mu_j)}{\Lambda_{\F}(1+2s+i\mu_j)} \int_{\gp{K}} f(\kappa) d\kappa \cdot \chi_1^{-1}(\det g) \norm[\det g]_{\ag{A}}^{\frac{i\mu_j}{2}}.
\label{RegEisDef}
\end{equation}
\noindent For any $\varphi \in \Aut^{\freg}(\GL_2,\omega)$ with auxiliary data given in Definition \ref{FinRegFuncDef} we define (\cite[(5.3)]{Wu9})
\begin{equation}
	\Reis(\varphi) = \sideset{}{_{\substack{\Re \alpha_j > 0 \\ \alpha_j \neq \frac{1}{2}+i\mu_j}}} \sum \frac{\partial^{n_j}}{\partial s^{n_j}}\eis(\alpha_j, f_j) + \sideset{}{_{\substack{\Re \alpha_j > 0 \\ \alpha_j = \frac{1}{2}+i\mu_j}}} \sum \frac{\partial^{n_j}}{\partial s^{n_j}}\eis^{\reg}(\alpha_j, f_j),
\label{ReisDef}
\end{equation}
\noindent where $\mu_j \in \ag{R}$ is defined only if $\omega^{-1}\chi_j^2(y)=\norm[y]_{\ag{A}}^{-2i\mu_j}$. This defines a linear map
	$$ \Aut^{\freg}(\GL_2,\omega) \to \Aut^{\freg}(\GL_2,\omega), \quad \varphi \mapsto \Reis(\varphi), $$
such that $\varphi - \Reis(\varphi) \in \intL^1(\GL_2,\omega)$, which is $\GL_2(\ag{A})$-intertwining when $\Ex(\varphi)$ does not contain $\norm_{\ag{A}}$. We denote the image by $\Reis(\GL_2, \omega)$. Moreover if $\Ex^+(\varphi) \cap \Ex^-(\varphi) = \emptyset$ then $\Reis(\varphi)$ is the unique element in $\Reis(\GL_2, \omega)$ such that $\varphi - \Reis(\varphi) \in \intL^2(\GL_2,\omega)$ (\cite[Proposition 5.25]{Wu9}), and we call it the $\intL^2$-residue of $\varphi$ (\cite[Definition 5.26]{Wu9}). In the case $\omega = 1$, $\Aut^{\freg}(\GL_2,1)$ is in the range of applicability of the regularized integral and (\citep[Proposition 5.27]{Wu9})
	$$ \int_{[\PGL_2]}^{\reg} \varphi(g) dg = \int_{[\PGL_2]} (\varphi(g) - \Reis(\varphi)(g)) dg. $$
In particular the above equation proves the $\GL_2(\ag{A})$-invariance of the regularized integral as a functional on $\Aut^{\freg}(\GL_2,1)$, when $\Ex(\varphi)$ does not contain $\norm_{\ag{A}}$. In this case the above equality was originally due to Zagier \cite{Za82}. We carefully generalized in \cite[Theorem 5.12 \& Definition 5.13]{Wu9} this theory into the adelic setting and proved the above equality without constraint on $\Ex(\varphi)$.

	In view of the inclusion (\cite[Remark 5.19]{Wu9})
	$$ \Aut^{\freg}(\GL_2,\omega_1) \cdot \Aut^{\freg}(\GL_2,\omega_2) \subset \Aut^{\freg}(\GL_2,\omega_1 \omega_2), $$
we can consider the following bilinear form. Let $\pi_j, j=1,2$ be two principal series representations with central character $\omega_j$ satisfying $\omega_1 \omega_2 = 1$. Let $V_j$ be the vector space of $\pi_j$ realized in the induced model from $\grB(\ag{A})$ with subspace of smooth vectors $V_j^{\infty}$. We then get a $\GL_2(\ag{A})$-invariant bilinear form
	$$ V_1^{\infty} \times V_2^{\infty} \to \ag{C}, \quad (f_1,f_2) \mapsto \int_{[\PGL_2]}^{\reg} \eis(f_1)(g) \eis(f_2)(g) dg, $$
where $\eis(f_j)$ should be suitably regularized if $\pi_j$ is at a position which creates a pole/zero for the relevant Eisenstein series. We succeeded in \cite[Theorem 6.5]{Wu9} to identify this bilinear form in the induced model. In order to present the result, we need to introduce some extra notations. Precisely, if we identify for any $s \in \ag{C}$ the space of functions $\pi_s$ with $H$, where
	$$ \pi_s := \Ind_{\gp{B}(\ag{A})}^{\GL_2(\ag{A})} (\norm_{\ag{A}}^s, \norm_{\ag{A}}^{-s}), \quad H := \Ind_{\gp{B}(\ag{A}) \cap \gp{K}}^{\gp{K}} 1, $$
then we can regard the intertwining operator $\Intw_s : \pi_s \to \pi_{-s}$ as a map from $H$ to itself. Using the \emph{flat section} map $H \to \pi_s, f \mapsto f_s$, we mean
	$$ (\Intw_s f_s)(a(y)\kappa) =: \norm[y]_{\ag{A}}^{\frac{1}{2}-s} (\Intw_s f)(\kappa), \quad \text{i.e.,} \quad \Intw_s f_s = \left( \Intw_s f \right)_{-s}. $$
	Let $e_0 \in H$ be the constant function taking value $1$. Define
	$$ \ProjP_{\gp{K}}: H \to \ag{C}, \quad f \mapsto \int_{\gp{K}} f(\kappa) d\kappa, $$
where $d\kappa$ is the probability Haar measure on $\gp{K}$. We obtain a map from $H$ to itself
	$$ \widetilde{\Intw}_s := \Intw_s \circ (I - \ProjP_{\gp{K}}e_0), $$
where $I$ is the identity map. Since $\Intw_s$ is ``diagonalizable'', we obtain the Taylor expansion as operators
	$$ \Intw_sf = \sum_{n=0}^{\infty} \frac{s^n}{n!} \Intw_0^{(n)}f, \quad \text{resp.} \quad \widetilde{\Intw}_{1/2+s} f = \sum_{n=0}^{\infty} \frac{s^n}{n!} \widetilde{\Intw}_{1/2}^{(n)}f. $$
\begin{theorem}{(\cite[Theorem 6.5]{Wu9})}
	The regularized integral of the product of two unitary Eisenstein series is computed as:
\begin{itemize}
	\item[(1)] If $\pi_1 = \pi(\xi_1,\xi_2), \pi_2 = \pi(\xi_1^{-1}, \xi_2^{-1})$ resp. $\pi_2 = \pi(\xi_2^{-1}, \xi_1^{-1})$ and $\xi_1 \neq \xi_2$, then
	$$ \int_{[\PGL_2]}^{\reg} \eis(0,f_1) \eis(0,f_2) = \frac{2\lambda_{\F}^{(0)}(0)}{\lambda_{\F}^{(-1)}(0)} \ProjP_{\gp{K}}(f_1f_2) - \ProjP_{\gp{K}}(\Intw_0^{(1)}f_1 \cdot \Intw_0 f_2), \quad \text{resp.} $$
	$$ \frac{\lambda_{\F}^{(0)}(0)}{\lambda_{\F}^{(-1)}(0)} (\ProjP_{\gp{K}}(f_1 \Intw_0 f_2) + \ProjP_{\gp{K}}(f_2 \Intw_0 f_1)) - \ProjP_{\gp{K}}(\Intw_0^{(1)}f_1 \cdot f_2). $$
	\item[(2)] If $\pi_1 = \pi(\xi,\xi), \pi_2 = \pi(\xi^{-1},\xi^{-1})$, then
\begin{align*}
	&\quad \int_{[\PGL_2]}^{\reg} \eis^{(1)}(0,f_1) \eis^{(1)}(0,f_2) = \frac{4\lambda_{\F}^{(2)}(0)}{\lambda_{\F}^{-1}(0)} \ProjP_{\gp{K}}(f_1f_2) + \frac{4\lambda_{\F}^{(2)}(0)}{\lambda_{\F}^{-1}(0)} \ProjP_{\gp{K}}(f_1 \cdot \Intw_0^{(1)} f_2) \\
	&\quad + \frac{\lambda_{\F}^{(0)}(0)}{\lambda_{\F}^{-1}(0)} \ProjP_{\gp{K}}(\Intw_0^{(1)}f_1 \cdot \Intw_0^{(1)}f_2) - \frac{1}{3} \ProjP_{\gp{K}}(\Intw_0^{(3)}f_1 \cdot f_2) - \ProjP_{\gp{K}}(\Intw_0^{(2)}f_1 \cdot \Intw_0^{(1)}f_2).
\end{align*}
\end{itemize}
	Here we have written (\cite[(5.2)]{Wu9})
	$$ \lambda_{\F}(s) := \frac{\Lambda_{\F}(-2s)}{\Lambda_{\F}(2+2s)} = \frac{\lambda_{\F}^{(-1)}(0)}{s} + \sum_{n=0}^{\infty} \frac{s^n}{n!} \lambda_{\F}^{(n)}(0). $$
\label{RIPEisUnitary}
\end{theorem}

	\subsection{Regularized Triple Product Formula}
	\label{RTPF}
	
	Let's first complete the analysis for products of two Eisenstein series. We recall a lemma.
\begin{lemma}{(\cite[Lemma 6.4]{Wu9})}
	Let $f, f_1, f_2 \in \Res_{\gp{K}}^{\GL_2(\ag{A})} \pi(1,1)$. For $0 \neq s \in \ag{C}$ small, we have for any $n, n_1, n_2 \in \ag{N}$
	$$ \int_{[\PGL_2]}^{\reg} \eis^{\reg, (n)}(\frac{1}{2}+s, f) = -\frac{\lambda_{\F}^{(n)}(s)}{\lambda_{\F}^{(-1)}(0)} \ProjP_{\gp{K}}(f); $$
	$$ \int_{[\PGL_2]}^{\reg} \eis^{\reg, (n_1)}(\frac{1}{2}+s, f_1) \eis^{\reg, (n_2)}(\frac{1}{2}, f_2) = 0. $$
\label{SimpleProdSing}
\end{lemma}
\noindent We also recall the technique of \emph{deformation} (\cite[(6.1)]{Wu9}), inspired by the work of Michel \& Venkatesh. In general, if $\varphi \in \Aut^{\freg}(\PGL_2), \Reis \in \Reis(\PGL_2)$ are given, so that $\varphi - \Reis \in \intL^1([\PGL_2])$, and if we can find continuous families $\varphi_s \in \Aut^{\freg}(\PGL_2), \Reis_s \in \Reis(\PGL_2)$ which coincide with $\varphi, \Reis$ at $s=0$, then we have
\begin{equation}
	\int_{[\PGL_2]}^{\reg} \varphi = \int_{[\PGL_2]} \varphi - \Reis = \lim_{s \to 0} \int_{[\PGL_2]} \varphi_s - \Reis_s = \lim_{s \to 0} \left( \int_{[\PGL_2]}^{\reg} \varphi_s - \int_{[\PGL_2]}^{\reg} \Reis_s \right).
\label{DeformTec}
\end{equation}

	We turn to the study of regularized integrals of the form
	$$ \int_{[\PGL_2]}^{\reg} \eis(\frac{1}{2},f_1) \eis(\frac{1}{2},f_2). $$
	Denote $e=e_1$. We can write
	$$ \eisCst^{\reg}(s,f) = f_s + (\widetilde{\Intw}_s f)_{-s} + \lambda_{\F}(s-\frac{1}{2}) \ProjP_{\gp{K}}(f) \left( e_{-s} - e_{-1/2} \right); $$
\begin{align*}
	\eisCst^{\reg,(n)}(\frac{1}{2},f) &= f_{1/2}^{(n)} + \sum_{k=0}^n \binom{n}{k} (-1)^k (\widetilde{\Intw}_{1/2}^{(n-k)} f)_{-1/2}^{(k)} \\
	&+ \ProjP_{\gp{K}}(f) \cdot \left\{ \frac{(-1)^{n+1} \lambda_{\F}^{-1)}(0)}{n+1} e_{-1/2}^{(n+1)} + \sum_{k=1}^n \binom{n}{k} (-1)^k \lambda_{\F}^{(n-k)}(0) e_{-1/2}^k \right\},
\end{align*}
from which one easily deduce $\eisCst^{\reg}(1/2+s,f_1) \eisCst^{\reg,(n_2)}(1/2,f_2)$. We tentatively define
\begin{align*}
	\Reis^{\reg}(s) &:= \eis^{(n_2)}(\frac{3}{2}+s, f_1f_2) + \sum_{k=0}^{n_2} \binom{n_2}{k} (-1)^k \eis^{\reg, (k)}(\frac{1}{2}+s, f_1 \widetilde{\Intw}_{1/2}^{(n_2-k)} f_2) \\
	&\quad + \ProjP_{\gp{K}}(f_2) \cdot \left\{ \frac{(-1)^{n_2+1} \lambda_{\F}^{(-1)}(0)}{n_2+1} \eis^{\reg,(n_2+1)}(\frac{1}{2}+s, f_1) \right. \\
	&\quad \left. + \sum_{k=1}^{n_2} \binom{n_2}{k} (-1)^k \lambda_{\F}^{(n_2-k)}(0) \eis^{\reg,(k)}(\frac{1}{2}+s, f_1) \right\} \\
	&\quad + \eis^{\reg,(n_2)}(\frac{1}{2}-s, f_2 \widetilde{\Intw}_{1/2+s}f_1) \\
	&\quad + \lambda_{\F}(s) \ProjP_{\gp{K}}(f_1) \cdot \left\{ \eis^{\reg,(n_2)}(\frac{1}{2}-s, f_2) - \eis^{\reg,(n_2)}(\frac{1}{2}, f_2) \right\}.
\end{align*}
	Applying Lemma \ref{SimpleProdSing} with $n_1=0$ together with (\ref{DeformTec}), we get
\begin{align*}
	&\quad \int_{[\PGL_2]}^{\reg} \eis^{\reg}(\frac{1}{2},f_1) \eis^{\reg,(n_2)}(\frac{1}{2},f_2) = \lim_{s \to 0} \int_{[\PGL_2]} \eis^{\reg}(\frac{1}{2}+s,f_1) \eis^{\reg,(n_2)}(\frac{1}{2},f_2) - \Reis^{\reg}(s) \\
	&= \lim_{s \to 0} \sum_{k=0}^{n_2} \binom{n_2}{k} \frac{(-1)^k}{\lambda_{\F}^{(-1)}(0)} \lambda_{\F}^{(k)}(s) \ProjP_{\gp{K}}(f_1 \widetilde{\Intw}_{1/2}^{(n_2-k)} f_2) + \frac{\lambda_{\F}^{(n_2)}(-s)}{\lambda_{\F}^{(-1)}(0)} \ProjP_{\gp{K}}(f_2 \widetilde{\Intw}_{1/2+s}f_1) \\
	&\quad + \ProjP_{\gp{K}}(f_1) \ProjP_{\gp{K}}(f_2) \cdot \left\{ \frac{(-1)^{n_2+1} \lambda_{\F}^{(n_2+1)}(s)}{n_2+1} + \sum_{k=1}^{n_2} \binom{n_2}{k} (-1)^k \frac{\lambda_{\F}^{(n_2-k)}(0) \lambda_{\F}^{(k)}(s)}{\lambda_{\F}^{(-1)}(0)} + \lambda_{\F}(s) \lambda_{\F}^{(n_2)}(-s) \right\}.
\end{align*}
	Taking Laurent expansions, we verify that the function in $s$ in the range of the above limit is regular at $s=0$, unlike its appearance. The symmetry
	$$ \ProjP_{\gp{K}}(f_1 \widetilde{\Intw}_{1/2}^{(k)}f_2) = \ProjP_{\gp{K}}(f_2 \widetilde{\Intw}_{1/2}^{(k)}f_1), \forall k \in \ag{N} $$
must be used. Moreover, it can be differentiated $n_1$ times to deduce (3) of the following:

\begin{theorem}
\begin{itemize}
	\item[(1)] If $\pi_1 \not\simeq \widetilde{\pi}_2$ and $\xi_1 = \xi_1', \xi_2 \neq \xi_2'$ resp. $\xi_1 \neq \xi_1', \xi_2 = \xi_2'$ resp. $\xi_1=\xi_1', \xi_2 = \xi_2'$, $\xi_1\xi_2 \neq 1$ and $\xi_1^2\xi_2^2 = 1$, then for any $n_1,n_2 \in \ag{N}$
	$$ \int_{[\PGL_2]}^{\reg} \eis^{\reg,(n_1)}(\frac{1}{2},f_1) \cdot \eis^{(n_2)}(\frac{1}{2},f_2) = 0 \quad \text{resp.} \quad \int_{[\PGL_2]}^{\reg} \eis^{(n_1)}(\frac{1}{2},f_1) \cdot \eis^{\reg,(n_2)}(\frac{1}{2},f_2) = 0 $$
	$$ \text{resp.} \quad  \int_{[\PGL_2]}^{\reg} \eis^{\reg,(n_1)}(\frac{1}{2},f_1) \cdot \eis^{\reg,(n_2)}(\frac{1}{2},f_2) = 0. $$
	\item[(2)] If $\pi_1 = \pi(\xi_1,\xi_2), \pi_2 = \pi(\xi_1^{-1}, \xi_2^{-1})$ resp. $\pi_2 = \pi(\xi_2^{-1}, \xi_1^{-1})$ with $\xi_1 \neq \xi_2$, then for any $n_1,n_2 \in \ag{N}$
	$$ \int_{[\PGL_2]}^{\reg} \eis^{(n_1)}(\frac{1}{2},f_1) \cdot \eis^{(n_2)}(\frac{1}{2},f_2) = 0, \quad \text{resp.} $$
is a linear combination with coefficients depending only on $n_1,n_2$ and $\lambda_{\F}(s)$ of
	$$ \ProjP_{\gp{K}}(\Intw_{1/2}^{(n_1+n_2+1)}f_1 \cdot f_2); \quad \ProjP_{\gp{K}}(\Intw_{1/2}^{(l)}f_1 \cdot f_2) = \ProjP_{\gp{K}}(f_1 \cdot \Intw_{1/2}^{(l)} f_2), 0 \leq l \leq \max(n_1,n_2). $$
	\item[(3)] If $\pi_1 = \pi(\xi,\xi), \pi_2 = \pi(\xi^{-1},\xi^{-1})$, then for any $n_1,n_2 \in \ag{N}$
	$$ \int_{[\PGL_2]}^{\reg} \eis^{\reg,(n_1)}(\frac{1}{2},f_1) \cdot \eis^{\reg,(n_2)}(\frac{1}{2},f_2) $$
is a linear combination with coefficients depending only on $n_1,n_2$ and $\lambda_{\F}(s)$ of
	$$ \ProjP_{\gp{K}}(\widetilde{\Intw}_{1/2}^{(l)}f_1 \cdot f_2) = \ProjP_{\gp{K}}(f_1 \cdot \widetilde{\Intw}_{1/2}^{(l)} f_2), 0 \leq l \leq \max(n_1,n_2); $$
	$$ \ProjP_{\gp{K}}(\widetilde{\Intw}_{1/2}^{(n_1+n_2+1)}f_1 \cdot f_2); \quad \ProjP_{\gp{K}}(f_1) \ProjP_{\gp{K}}(f_2). $$
\end{itemize}
\label{RIPEisSing}
\end{theorem}

	Next, we give some complement of the main theorem of regularized integral \cite[Theorem 5.12]{Wu9}. Let $\xi_1,\xi_2, \omega$ be Hecke characters with $\xi_1 \xi_2 \omega =1$. Let $f \in \pi(\xi_1,\xi_2)$ and $\varphi \in \Cont^{\infty}(\GL_2, \omega)$, i.e., a smooth function on $\GL_2(\F) \backslash \GL_2(\ag{A})$ with central character $\omega$. Suppose $\varphi$ is \emph{finitely regularizable} defined in Definition \ref{FinRegFuncDef}.
\begin{proposition}
	For $\Re s \gg 1$ sufficiently large,
	$$ R(s,\varphi; f) := \int_{\F^{\times} \backslash \ag{A}^{\times}} \int_{\gp{K}} (\varphi_{\gp{N}} - \varphi_{\gp{N}}^*)(a(y)\kappa) f(\kappa) \xi_1(y) \norm[y]_{\ag{A}}^{s-\frac{1}{2}} d\kappa d^{\times}y $$
is absolutely convergent. It has a meromorphic continuation to $s \in \ag{C}$. If in addition
	$$ \Theta := \max_j \{ \Re \alpha_j \} < 0, $$
then we have, with the right hand side absolutely converging
	$$ R(s,\varphi; f) = \int_{[\PGL_2]} \varphi \cdot \eis(s,f), \quad \Theta < \Re s < -\Theta.  $$
	In the above region, the possible poles of $R(s, \varphi; f)$ are
\begin{itemize}
	\item $1/2 + i\mu(\xi_1\xi_2^{-1})$ if $\xi_1 \xi_2^{-1}$ is trivial on $\ag{A}^{(1)}$;
	\item $(\rho - 1)/2$ where $\rho$ runs over the non-trivial zeros of $L(s,\xi_1\xi_2^{-1})$.
\end{itemize}
	In particular $R(s,\varphi; f)$ is holomorphic for $0 \leq \Re s < \min(-\Theta,1/2)$.
\label{VRI}
\end{proposition}
\begin{proof}
	The proof is quite similar to that of \cite[Theorem 5.12 (3)]{Wu9}, except that $\Intw f_s$ is no longer explicitly computable. In fact, we have for $T > 1, \Re s \gg 1$, using the standard Rankin-Selberg unfolding
\begin{align*}
	\int_{[\PGL_2]} \varphi \cdot \Lambda^T \eis(s,f) &= R(s,\varphi; f) \\
	&- \int_{\F^{\times} \backslash \ag{A}^{\times}} \left( \int_{\gp{K}} (\varphi_{\gp{N}} - \varphi_{\gp{N}}^*)(a(y)\kappa) f(\kappa) d\kappa \right) \xi_1(y) \norm[y]_{\ag{A}}^{s-1/2} 1_{\norm[y]_{\ag{A}} > T} d^{\times} y \\
	&- \int_{\F^{\times} \backslash \ag{A}^{\times}} \left( \int_{\gp{K}} (\varphi_{\gp{N}} - \varphi_{\gp{N}}^*)(a(y)\kappa) \Intw f_s(\kappa) d\kappa \right) \xi_2(y) \norm[y]_{\ag{A}}^{-s-1/2} 1_{\norm[y]_{\ag{A}} > T} d^{\times} y \\
	&+ \Vol(\F^{\times} \backslash \ag{A}^{(1)}) \left( \sum_{j=1}^l \int_{\gp{K}} f_j(\kappa) f(\kappa) d\kappa \cdot 1_{\chi_j \xi_1(\ag{A}^{(1)}) = 1} \cdot \frac{1}{n_j !} \frac{\partial^{n_j}}{\partial s^{n_j}} \left( \frac{T^{s+\alpha_j + i\mu_j}}{s+\alpha_j + i\mu_j} \right) \right. \\
	&- \left. \sum_{j=1}^l \int_{\gp{K}} f_j(\kappa) \Intw f_s(\kappa) d\kappa \cdot 1_{\chi_j \xi_2(\ag{A}^{(1)}) = 1} \cdot \frac{(-1)^{n_j}}{n_j !} \frac{\partial^{n_j}}{\partial s^{n_j}} \left( \frac{T^{-s+\alpha_j + i\mu_j'}}{-s+\alpha_j + i\mu_j'} \right) \right),
\end{align*}
where $\mu_j$ resp. $\mu_j'$ is such that
	$$ \chi_j \xi_1(y) = \norm[y]_{\ag{A}}^{i\mu_j}, \quad \text{resp.} \quad \chi_j \xi_2(y) = \norm[y]_{\ag{A}}^{i\mu_j'}. $$
	We conclude by first shifting $s$ to the desired region, then letting $T \to \infty$. The possible poles are encoded in the possible poles of $\Intw f_s$, which are included in those of $L(1+2s, \xi_1\xi_2^{-1})^{-1}$ in the above region (c.f. for example \cite[Corollary 3.7, 3.10 \& Lemma 3.18]{Wu5}).
\end{proof}
\begin{proposition}
	Let notations be as in the previous proposition with $\Theta \leq -1/2$. Recall
	$$ \varphi_{\gp{N}}^*(n(x)a(y)k)=\varphi_{\gp{N}}^*(a(y)k)=\sum_{j=1}^l \chi_j(y) \norm[y]_{\ag{A}}^{\frac{1}{2}+\alpha_j} \log^{n_i} \norm[y]_{\ag{A}} f_j(k). $$
\begin{itemize}
	\item[(1)] If $\xi_1 \neq \xi_2$, then
	$$ \left( \frac{\partial^n R}{\partial s^n} \right)^{\hol}(\frac{1}{2},\varphi;f) = \int_{[\PGL_2]}^{\reg} \varphi \cdot \eis^{(n)}(\frac{1}{2},f) - \sideset{}{_j'}\sum \frac{\lambda_{\F}^{(n+n_j)}(0)}{\lambda_{\F}^{(-1)}(0)} \ProjP_{\gp{K}}(f_j f), $$
	where the summation is over $j$ such that $\xi_1 \chi_j (\ag{A}^{(1)}) =1, \alpha_j + i \mu(\xi_1\chi_j) = -1/2$.
	\item[(2)] If $\xi_1 = \xi_2 = \xi$, then
\begin{align*}
	\left( \frac{\partial^n R}{\partial s^n} \right)^{\hol}(\frac{1}{2},\varphi;f) &= \int_{[\PGL_2]}^{\reg} \varphi \cdot \eis^{\reg,(n)}(\frac{1}{2},f) - \sideset{}{_j'}\sum \frac{\lambda_{\F}^{(n+n_j)}(0)}{\lambda_{\F}^{(-1)}(0)}  \ProjP_{\gp{K}}(f_j f) \\
	&+ \lambda_{\F}^{(n)}(0) \cdot \ProjP_{\gp{K}}(f \cdot (\xi^{-1} \circ \det)) \cdot \int_{[\PGL_2]} \varphi \cdot (\xi \circ \det),
\end{align*}
	where the summation over $j$ is as in the previous case.
\end{itemize}
\label{TripleToDouble}
\end{proposition}
\begin{proof}
	The case (1) being simpler, we only give details for (2). By twisting, we may assume $\xi=1$. Let $s$ be small with $\Re s < 0$. The $\intL^2$-residue of $\varphi \cdot \eis(1/2+s,f)$ is given by
	$$ \Reis(s) := \sideset{}{_j} \sum \eis^{(n_j)}(s+1+\alpha_j, f_j f), $$
	where the summation is over $j$ such that $\Re \alpha_j > -1$. Define
	$$ \Reis^{\reg}(s) := \sideset{}{_j'} \sum \eis^{\reg,(n_j)}(s+1+\alpha_j, f_j f) + \sideset{}{_j^*} \sum \eis^{(n_j)}(s+1+\alpha_j, f_j f) $$
where $\sideset{}{_j'} \sum$ is the summation as in the statement and $\sideset{}{_j^*} \sum$ is the rest. By the previous proposition, we have
\begin{align*}
	R(\frac{1}{2}+s,\varphi;f) &= \int_{[\PGL_2]} \varphi \cdot \eis(\frac{1}{2}+s,f) = \int_{[\PGL_2]} \varphi \cdot \eis^{\reg}(\frac{1}{2}+s,f) + \lambda_{\F}(s) \ProjP_{\gp{K}}(f) \cdot \int_{[\PGL_2]} \varphi \\
	&= \int_{[\PGL_2]} \left( \varphi \cdot \eis^{\reg}(\frac{1}{2}+s,f) - \Reis^{\reg}(s) \right) - \int_{[\PGL_2]} \left( \Reis(s) - \Reis^{\reg}(s) \right) \\
	&\quad + \lambda_{\F}(s) \ProjP_{\gp{K}}(f) \cdot \int_{[\PGL_2]} \varphi.
\end{align*}
	Since $\Reis^{\reg,(n)}(s)$ is the $\intL^2$-residue of $\varphi \cdot \eis^{\reg,(n)}(\frac{1}{2}+s,f)$, we can compare the finite parts of both sides and conclude by
	$$ \Reis(s) - \Reis^{\reg}(s) = \sideset{}{_j'} \sum \lambda_{\F}^{(n_j)}(s) \ProjP_{\gp{K}}(f_jf). $$
\end{proof}

	Finally, we state and prove a special case of regularized triple product formulas. The method used in the proof is applicable in any general case but on the one treated in the following theorem is used in the current paper.
\begin{theorem}
	Let $f_j \in \pi(1,1), j=1,2,3$. Then for any $n \in \ag{N}$
	$$ \int_{[\PGL_2]}^{\reg} \eis^*(0,f_1) \cdot \eis^*(0,f_2) \cdot \eis^{\reg,(n)}(\frac{1}{2},f_3) $$
is the sum of
	$$ \left( \frac{\partial^n R}{\partial s^n} \right)^{\hol} \left( \frac{1}{2}, \eis^*(0,f_1) \cdot \eis^*(0,f_2); f_3 \right) $$
and a weighted sum with coefficients depending only on $\lambda_{\F}(s)$ of
	$$ \ProjP_{\gp{K}}(\Intw_0^{(l)}f_1 \cdot f_2) \ProjP_{\gp{K}}(f_3), \quad 0 \leq l \leq 3; $$
	$$ \ProjP_{\gp{K}}(f_1\cdot f_2 \cdot \widetilde{\Intw}_{1/2}^{(l)} f_3), \quad 0 \leq l \leq \max(2,n) \quad \& \quad l = n+3; $$
	$$ \ProjP_{\gp{K}}((f_1 \Intw_0 f_2 + f_2 \Intw_0 f_1) \cdot \widetilde{\Intw}_{1/2}^{(l)} f_3), \quad 0 \leq l \leq \max(1,n) \quad \& \quad l = n+2; $$
	$$ \ProjP_{\gp{K}}(\Intw_0 f_1\cdot \Intw_0 f_2 \cdot \widetilde{\Intw}_{1/2}^{(l)} f_3), \quad 0 \leq l \leq n \quad \& \quad l = n+1. $$
\label{RegTripEis}
\end{theorem}
\begin{proof}
	We shall only point out how the computation is effectuated, since the precise formulas are quite long, useless for the purpose of the current paper and would only obscure the idea.
\begin{align*}
	\Reis(f_1,f_2) &:= (\Lambda_{\F}^*)^2 \cdot \left\{ \eis^{\reg,(2)}(\frac{1}{2},f_1f_2) + \frac{1}{2} \eis^{\reg,(1)}(\frac{1}{2}, f_1 \cdot \Intw_0^{(1)}f_2 + \Intw_0^{(1)}f_1 \cdot f_2) \right. \\
	&\quad\left. + \frac{1}{4} \eis^{\reg}(\frac{1}{2}, \Intw_0^{(1)}f_1 \cdot \Intw_0^{(1)}f_2) \right\}
\end{align*}
is the $\intL^2$-residue of $\eis^*(0,f_1) \cdot \eis^*(0,f_2)$. Let $\varphi := \eis^*(0,f_1) \cdot \eis^*(0,f_2) - \Reis(f_1,f_2)$, then we need to compute
	$$ \int_{[\PGL_2]}^{\reg} \varphi \cdot \eis^{\reg,(n)}(\frac{1}{2},f_3) + \int_{[\PGL_2]}^{\reg} \Reis(f_1,f_2) \cdot \eis^{\reg,(n)}(\frac{1}{2},f_3). $$
	The first term is computed by Proposition \ref{TripleToDouble} (2), involving
	$$ \int_{[\PGL_2]} \varphi = \int_{[\PGL_2]}^{\reg} \eis^*(0,f_1) \cdot \eis^*(0,f_2) = \frac{(\Lambda_{\F}^*)^2}{4} \int_{[\PGL_2]}^{\reg} \eis^{(1)}(0,f_1) \cdot \eis^{(1)}(0,f_2), $$
which is treated in Theorem \ref{RIPEisUnitary} (2). The second term is treated in Theorem \ref{RIPEisSing} (3).
\end{proof}

	\subsection{Extension of Global Zeta Integral}
	
	Fixing a central Hecke character $\omega$ over a number field $\F$, we extend the global part of the Hecke-Jacquet-Langlands' theory to $\Aut^{\freg}(\GL_2,\omega)$ Definition \ref{FinRegFuncDef}, as well as an analogue of ``approximate functional equation''. Note that Eisenstein series are in $\Aut^{\freg}(\GL_2,\omega)$.
\begin{definition}
	Let $\varphi \in \Aut^{\freg}(\GL_2,\omega)$ for some Hecke character $\omega$. For a Hecke character $\chi$ and $s \in \C, \Re s \gg 1$, we define the zeta-functional by
	$$ \zeta(s,\chi,\varphi) = \int_{\F^{\times} \backslash \A^{\times}} (\varphi-\varphi_{\grN}^*)(a(y)) \chi(y) \norm[y]_{\A}^{s-\frac{1}{2}} d^{\times}y, $$
where we recall the \emph{essential constant term}
	$$ \varphi_{\gp{N}}^*(n(x)a(y)k)=\varphi_{\gp{N}}^*(a(y)k)=\sum_{i=1}^l \chi_i(y) \norm[y]_{\ag{A}}^{\frac{1}{2}+\alpha_i} \log^{n_i} \norm[y]_{\ag{A}} f_i(k). $$
\label{ZetaFunctDef}
\end{definition}
\begin{proposition}
	$\zeta(s,\chi,\varphi)$ has a meromorphic continuation to $s \in \C$ with functional equation
	$$ \zeta(s,\chi,\varphi) = \zeta(1-s,\omega^{-1}\chi^{-1},w.\varphi). $$
	It has possible poles at $s=-\alpha_j-i\mu(\chi_j\chi)$ with $\chi_j\chi(\ag{A}^{(1)}) = 1$ resp. $s=1+\alpha_j+i\mu(\chi_j\omega^{-1}\chi^{-1})$ with $\chi_j\omega^{-1}\chi^{-1}(\ag{A}^{(1)}) = 1$, with pure order $n_j+1$. Here $\mu(\chi) \in \ag{R}$ is defined for $\chi(\ag{A}^{(1)})=1$ as $\chi(t) = \norm[t]_{\ag{A}}^{i\mu(\chi)}$.
\label{GlobZetaFR}
\end{proposition}
\begin{proof}
	By the invariance of $\varphi$ at left by $w$, we can re-write the zeta-integral as
\begin{align*}
	\zeta(s,\chi,\varphi) &= \int_{\substack{y \in \F^{\times} \backslash \A^{\times} \\ \norm[y]_{\A} \geq 1}} (\varphi-\varphi_{\grN}^*)(a(y)) \chi(y) \norm[y]_{\A}^{s-\frac{1}{2}} d^{\times}y \\
	&\quad + \int_{\substack{y \in \F^{\times} \backslash \A^{\times} \\ \norm[y]_{\A} \leq 1}} (w.\varphi-w.\varphi_{\grN}^*)(a(y^{-1})) \omega\chi(y) \norm[y]_{\A}^{s-\frac{1}{2}} d^{\times}y \\
	&\quad - \int_{\substack{y \in \F^{\times} \backslash \A^{\times} \\ \norm[y]_{\A} \leq 1}} \varphi_{\grN}^*(a(y)) \chi(y)\norm[y]_{\A}^{s-\frac{1}{2}} d^{\times}y \\
	&\quad + \int_{\substack{y \in \F^{\times} \backslash \A^{\times} \\ \norm[y]_{\A} \leq 1}} w.\varphi_{\grN}^*(a(y^{-1})) \omega\chi(y)\norm[y]_{\A}^{s-\frac{1}{2}} d^{\times}y.
\end{align*}
	We can calculate the integral concerning $\varphi_{\grN}^*$ and get
\begin{align*}
	\zeta(s,\chi,\varphi) &= \int_{\substack{y \in \F^{\times} \backslash \A^{\times} \\ \norm[y]_{\A} \geq 1}} (\varphi-\varphi_{\grN}^*)(a(y)) \chi(y) \norm[y]_{\A}^{s-\frac{1}{2}} d^{\times}y \\
	&\quad + \int_{\substack{y \in \F^{\times} \backslash \A^{\times} \\ \norm[y]_{\A} \geq 1}} (w.\varphi-w.\varphi_{\grN}^*)(a(y)) \omega^{-1}\chi^{-1}(y) \norm[y]_{\A}^{\frac{1}{2}-s} d^{\times}y + \zeta_{\F}^*(1) \cdot \\
	&\quad \left( \sum_{\chi_j\chi \mid_{\A^{(1)}} =1} \frac{(-1)^{n_j+1} f_j(1)}{(s+\alpha_j + i\mu(\chi_j\chi))^{n_j+1}} + \sum_{\chi_j\omega^{-1}\chi^{-1} \mid_{\A^{(1)}} =1} \frac{(-1)^{n_j+1} f_j(w)}{(1-s+\alpha_j + i\mu(\chi_j\omega^{-1}\chi^{-1}))^{n_j+1}} \right),
\end{align*}
	from which we easily deduce all the assertions.
\end{proof}

	We turn to the special case $\varphi(g) = \eis^*(s_0,\xi,\omega\xi^{-1};f)(g)$ the usual completed Eisenstein series resp. $\eis^{\reg}(s_0,\xi,\omega\xi^{-1};f)(g)$ (\ref{RegEisDef}), for which the local computation is the same as for a cusp form. Note that in this case $\varphi_{\grN}^*=\varphi_{\grN}$, hence $\varphi(a(y))-\varphi_{\grN}^*(a(y)) = \Sigma_{\alpha \in \F^{\times}} W_{\varphi}(a(\alpha y))$.
\begin{proposition}
	Let $\varphi(g) = \eis^*(s_0,\xi,\omega\xi^{-1};f)(g)$ resp. $\eis^{\reg}(s_0,\xi,\omega\xi^{-1};f)(g)$ where $\xi,\omega$ are Hecke characters and $f \in \pi_{\xi,\omega\xi^{-1}}$. The zeta-functional has a decomposition as an Euler product in which only a finite number of terms are not equal to $1$:
\begin{align*}
	\zeta(s,\chi,\varphi) &= \Lambda(s+s_0, \xi\chi) \Lambda(s-s_0,\omega\xi^{-1}\chi) \cdot \\
	&\quad \prod_v \frac{L_v(1+2s_0,\omega_v^{-1}\xi_v^2)}{L_v(s+s_0,\xi_v\chi_v)L_v(s-s_0,\omega_v\xi_v^{-1}\chi_v)} \int_{\F_v^{\times}} W_{f_v}^{(s_0)}(a(y_v)) \norm[y_v]_v^{s-\frac{1}{2}} d^{\times}y_v \text{ resp.}
\end{align*}
\begin{align*}
	\zeta(s,\chi,\varphi) &= \frac{\Lambda(s+s_0, \xi\chi) \Lambda(s-s_0,\omega\xi^{-1}\chi)}{\Lambda(1+2s_0,\omega^{-1}\xi^2)} \cdot \\
	&\quad \prod_v \frac{L_v(1+2s_0,\omega_v^{-1}\xi_v^2)}{L_v(s+s_0,\xi_v\chi_v)L_v(s-s_0,\omega_v\xi_v^{-1}\chi_v)} \int_{\F_v^{\times}} W_{f_v}^{(s_0)}(a(y_v)) \norm[y_v]_v^{s-\frac{1}{2}} d^{\times}y_v.
\end{align*}
\label{EulerZetaEis}
\end{proposition}

The way given in the proof of Proposition \ref{GlobZetaFR} is not the only way of the analytic continuation of the global zeta functional. Another version of truncation on the integral is closely related to the classical approximate functional equation. Let $h_0$ be a smooth function supported in the inteval $[0,2)$, being equal to $1$ on $[0,1]$. For any $A>0$, we denote by $h_{0,A}$ the function $t \mapsto h_0(t/A)$. We then have for $\Re s \gg 1$
\begin{align}
	\zeta(\frac{1}{2}+s,\chi,\varphi) &= \int_{\F^{\times} \backslash \A^{\times}} (\varphi-\varphi_{\grN}^*)(a(y)) \chi(y) \norm[y]_{\A}^s (1-h_{0,A}(\norm[y]_{\A})) d^{\times}y \label{GZFundE} \\
	&\quad + \int_{\F^{\times} \backslash \A^{\times}} (w.\varphi-w.\varphi_{\grN}^*)(a(y)) \omega^{-1}\chi^{-1}(y) \norm[y]_{\A}^{-s} h_{0,A}(\norm[y]_{\A}^{-1}) d^{\times}y \nonumber \\
	&\quad - \int_{\F^{\times} \backslash \A^{\times}} \varphi_{\grN}^*(a(y)) \chi(y)\norm[y]_{\A}^s h_{0,A}(\norm[y]_{\A}) d^{\times}y \nonumber \\
	&\quad + \int_{\F^{\times} \backslash \A^{\times}} w.\varphi_{\grN}^*(a(y)) \omega^{-1}\chi^{-1}(y)\norm[y]_{\A}^{-s} h_{0,A}(\norm[y]_{\A}^{-1}) d^{\times}y. \nonumber
\end{align}
	For the last two lines, it is not hard to compute their analytic continuation using the form of $\varphi_{\grN}^*$ and the analytic continuation of the Mellin transform of $h_0$ as (since $h_0'$ is of compact support contained in $(1,2)$)
\begin{align*}
	\Mellin{h_0}(s) &= \int_0^{\infty} h_0(t) t^s \frac{dt}{t} = \frac{1}{s} - \frac{\Mellin{h_0'}(s+1)+1}{s} \\
	&= (-1)^N \prod_{j=0}^{N-1} (s+j)^{-1} \Mellin{h_0^{(N)}}(s+N), \forall N \in \N, s \in \C.
\end{align*}
\begin{remark}
	We also have, first for $\Re s \ll -1$ then for $s \in \C$
	$$ \Mellin{1-h_0}(s) = \int_0^{\infty} (1-h_0)(t) t^s \frac{dt}{t} = \frac{\Mellin{h_0'}(s+1)}{s} = -\Mellin{h_0}(s). $$
\end{remark}
\noindent Then the last two lines of (\ref{GZFundE}) are defined for $s \in \C$ as, writing $s_j=1/2+\alpha_j+i\mu(\chi_j\chi)$
\begin{align*}
	&\quad -\zeta_{\F}^*(1)\sum_{j=1}^l f_j(1) \delta_{\chi_j\chi} \frac{\partial^{n_j}}{\partial s^{n_j}} \left( A^{s+s_j} \Mellin{h_0}(s+s_j) \right) \\
	&= \zeta_{\F}^*(1)\sum_{j=1}^l f_j(1) \delta_{\chi_j\chi} \frac{(-1)^{n_j+1}}{(s+s_j)^{n_j+1}} - \zeta_{\F}^*(1)\sum_{j=1}^l f_j(1) \delta_{\chi_j\chi} \cdot \\
	&\quad \sum_{k=0}^{n_j} \binom{n_j}{k} (\log A)^{n_j-k+1} \int_0^{\infty} h_0'(t) t^{s+s_j} \log^k t dt \cdot \int_0^1 A^{\delta (s+s_j)} \delta^{n_j-k} d \delta - \\
	&\quad \zeta_{\F}^*(1)\sum_{j=1}^l f_j(1) \delta_{\chi_j\chi} \int_0^{\infty} h_0'(t) \left( \int_0^1 \delta^{n_j} t^{\delta (s+s_j)} d \delta \right) \log^{n_j+1}t dt,
\end{align*}
	and writing $s_j'=1/2+\alpha_j+i\mu(\chi_j\omega^{-1}\chi^{-1})$
\begin{align*}
	&\quad \zeta_{\F}^*(1) \sum_{j=1}^l f_j(w) \delta_{\chi_j\omega^{-1}\chi^{-1}} (-1)^{n_j} \cdot \frac{\partial^{n_j}}{\partial s^{n_j}} \left( A^{s-s_j'} \Mellin{h_0}(s-s_j') \right) \\
	&= \zeta_{\F}^*(1) \sum_{j=1}^l f_j(w) \delta_{\chi_j\omega^{-1}\chi^{-1}} \frac{(-1)^{n_j+1}}{(s_j'-s)^{n_j+1}} + \zeta_{\F}^*(1)\sum_{j=1}^l f_j(w) \delta_{\chi_j\omega^{-1}\chi^{-1}} (-1)^{n_j} \cdot \\
	&\quad \sum_{k=0}^{n_j} \binom{n_j}{k} (\log A)^{n_j-k+1} \int_0^{\infty} h_0'(t) t^{s-s_j'} \log^k t dt \cdot \int_0^1 A^{\delta (s-s_j')} \delta^{n_j-k} d \delta + \\
	&\quad \zeta_{\F}^*(1)\sum_{j=1}^l f_j(w) \delta_{\chi_j\omega^{-1}\chi^{-1}} (-1)^{n_j} \int_0^{\infty} h_0'(t) \left( \int_0^1 \delta^{n_j} t^{\delta (s-s_j')} d \delta \right) \log^{n_j+1}t dt.
\end{align*}
	We separate the terms in the sum $\Sigma_{j=1}^l$ according as $s_j=0$ and $s_j \neq 0$ resp. $s_j'=0$ and $s_j' \neq 0$. For $s_j=0$ resp. $s_j'=0$, the finite part at $s=0$ is bounded, with implied constants depending only on $\F,n_j,h_0$, by
	$$ \sum_{j=1}^l \delta_{\chi_j\chi} 1_{s_j=0} O(\norm[f_j(1) \log^{n_j+1}A]) \text{ resp. } \sum_{j=1}^l \delta_{\chi_j\omega^{-1}\chi^{-1}} 1_{s_j'=0} O(\norm[f_j(w) \log^{n_j+1}A]). $$
	For $s_j \neq 0$ resp. $s_j' \neq 0$, they are of size at $s=0$, with implied constants depending only on $\F,\alpha_j,n_j,h_0$ and an arbitrary $N \in \N$,
	$$ \sum_{j=1}^l \delta_{\chi_j\chi} 1_{s_j \neq 0} O\left( \frac{A^{\Re s_j} \norm[f_j(1) \log^{n_j}A]}{\norm[s_j]^N} \right) \text{ resp. } \sum_{j=1}^l \delta_{\chi_j\omega^{-1}\chi^{-1}} 1_{s_j' \neq 0} O \left( \frac{A^{-\Re s_j'} \norm[f_j(w) \log^{n_j}A]}{\norm[s_j']^N} \right). $$
	The first resp. second line of (\ref{GZFundE}) is supported in $\norm[y]_{\A} \in [A, \infty)$ resp. $[(2A)^{-1}, \infty)$, hence is well-defined for all $s \in \C$ by the rapid decay of $\varphi-\varphi_{\grN}^*$. For the second line at $s=0$, we can apply Mellin inversion to see
\begin{align*}
	&\quad \int_{\F^{\times} \backslash \A^{\times}} (w.\varphi-w.\varphi_{\grN}^*)(a(y)) \omega^{-1}\chi^{-1}(y) h_{0,A}(\norm[y]_{\A}^{-1}) d^{\times}y \\
	&= \zeta_{\F}^*(1) \int_{\Re s_1=c_1 \gg 1} A^{s_1} \zeta(\frac{1}{2}+s_1,\omega^{-1}\chi^{-1},w.\varphi) \Mellin{h_0}(s_1) \frac{ds_1}{2\pi i}
\end{align*}
	which is bounded, with implied constant depending only on $\F,h_0$ and an arbitrary $N \in \N$, as
	$$ A^{c_1} O \left( \int_{\Re s_1=c_1 \gg 1} \frac{\norm[\zeta(\frac{1}{2}+s_1,\omega^{-1}\chi^{-1},w.\varphi)]}{\prod_{m=0}^{N-1} \norm[s_1+m]} \frac{\norm[ds_1]}{2\pi} \right), $$
	where $c_1$ can be chosen as any real number such that the integral defining $\zeta(1/2+s_1,\omega^{-1}\chi^{-1},w.\varphi)$ is absolutely convergent for $\Re s_1 \geq c_1$. Similarly, we have for any $B>0$ that
\begin{align*}
	&\quad \int_{\F^{\times} \backslash \A^{\times}} (\varphi-\varphi_{\grN}^*)(a(y)) \chi(y) (1-h_{0,B})(\norm[y]_{\A}) d^{\times}y \\
	&= -\int_{\Re s_2=c_2 \gg 1} B^{-s_2} \zeta(\frac{1}{2}+s_2,\chi,\varphi) \Mellin{h_0}(-s_2) \frac{ds_2}{2\pi i}
\end{align*}
	is bounded, with implied constant depending only on $\F,h_0$ and an arbitrary $N \in \N$ as
	$$ B^{-c_2} O \left( \int_{\Re s_2=c_2 \gg 1} \frac{\norm[\zeta(\frac{1}{2}+s_2,\chi,\varphi)]}{\prod_{m=0}^{N-1} \norm[s_2+m]} \frac{\norm[ds_2]}{2\pi} \right), $$
	where $c_2$ can be chosen as any real number such that the integral defining $\zeta(1/2+s_2,\chi,\varphi)$ is absolutely convergent for $\Re s_2 \geq c_2$.
\begin{definition}
	For any function $h: \R_+ \to \C$ and any Hecke character $\chi$, we define the $h$-\textbf{truncated (zeta-)integral} on $\Aut^{\freg}(\GL_2,\omega)$ as
	$$ \zeta(h,\chi,\varphi) = \int_{\F^{\times} \backslash \A^{\times}} h(\norm[y]_{\A}) \chi(y) (\varphi-\varphi_{\grN}^*)(a(y)) d^{\times}y. $$
\label{h-TruncZeta}
\end{definition}
	As a summary, we have obtained:
\begin{proposition}
	Take $h_0$ as indicated in the beginning, some positive constants $0<A<B$ and define $h(t)=h_{0,B}(t)-h_{0,A}(t), t>0$. For any $\varphi \in \Aut^{\freg}(\GL_2,\omega)$ with $\chi_i,\alpha_i,n_i$ given in Definition \ref{ZetaFunctDef}, we write
	$$ s_j=1/2+\alpha_j+i\mu(\chi_j\chi) \text{ resp. } s_j'=1/2+\alpha_j+i\mu(\chi_j\omega^{-1}\chi^{-1}) $$
	if $\chi_j\chi$ resp. $\chi_j\omega^{-1}\chi^{-1}$ is trivial on $\A^{(1)}$, and $\mu$ defined in Proposition \ref{GlobZetaFR}. Then the difference
	$$ \zeta^{\hol}(\frac{1}{2},\chi,\varphi) - \zeta(h,\chi,\varphi) $$
	is bounded, with implied constants depending only on $\F,\alpha_j,n_j,h_0$ and an arbitrary $N \in \N$, as the sum of
\begin{itemize}
	\item[(1)] \textbf{Degenerate Polar Part:}
	$$ \sum_{j=1}^l \delta_{\chi_j\chi} 1_{s_j=0} O(\norm[f_j(1) \log^{n_j+1}A]) + \sum_{j=1}^l \delta_{\chi_j\omega^{-1}\chi^{-1}} 1_{s_j'=0} O(\norm[f_j(w) \log^{n_j+1}A]). $$
	\item[(2)] \textbf{Normal Polar Part:}
	$$ \sum_{j=1}^l \delta_{\chi_j\chi} 1_{s_j \neq 0} O\left( \frac{A^{\Re s_j} \norm[f_j(1) \log^{n_j}A]}{\norm[s_j]^N} \right) + \sum_{j=1}^l \delta_{\chi_j\omega^{-1}\chi^{-1}} 1_{s_j' \neq 0} O \left( \frac{A^{-\Re s_j'} \norm[f_j(w) \log^{n_j}A]}{\norm[s_j']^N} \right). $$
	\item[(3)] \textbf{Lower Part:}
	$$ A^{c_1} O \left( \int_{\Re s=c_1 \gg 1} \frac{\norm[\zeta(\frac{1}{2}+s,\omega^{-1}\chi^{-1},w.\varphi)]}{\prod_{m=0}^{N-1} \norm[s+m]} \frac{\norm[ds]}{2\pi} \right). $$
	\item[(4)] \textbf{Upper Part:}
	$$ B^{-c_2} O \left( \int_{\Re s=c_2 \gg 1} \frac{\norm[\zeta(\frac{1}{2}+s,\chi,\varphi)]}{\prod_{m=0}^{N-1} \norm[s+m]} \frac{\norm[ds]}{2\pi} \right). $$
\end{itemize}
	In (3) resp. (4), $c_1>0$ resp. $c_2>0$ is any real number such that the integral defining $\zeta(1/2+s,\omega^{-1}\chi^{-1},w.\varphi)$ resp. $\zeta(1/2+s,\chi,\varphi)$ is absolutely convergent for $\Re s \geq c_1$ resp. $\Re s \geq c_2$.
\label{ApproxFE}
\end{proposition}
\begin{remark}
	We will use the bound for the normal polar part in the case $s_j$ is bounded away from $0$.
\end{remark}

	\subsection{Classical Vectors in Spherical Series}
	
	Let $\F$ be a non-archimedean local field with uniformizer $\varpi$, absolute valuation $\norm$, valuation ring $\vo$ \& ideal $\vp$ and cardinality of the residue class field $q$. Denote by $\pi_s = \Ind_{\gp{B}(\F)}^{\GL_2(\F)}(\norm^s, \norm^{-s})$ the principal series representation of $\PGL_2(\F)$, where $s \in \ag{C}$. For $s \in i \ag{R}$, $\pi_s$ are unitary with the underlying Hilbert spaces identified with the same one 
	$$ \Res_{\gp{K}}^{\GL_2(\F)} \pi_s = \Ind_{\gp{B}(\F) \cap \gp{K}}^{\gp{K}} (1,1) =: H, \quad \gp{K} = \GL_2(\vo). $$
	We regard $\pi_s, s \in \ag{C}$ as a family of representations of $\GL_2(\F)$ on $H$. By the Branching law, there is a canonical decomposition as $\gp{K}$-representations
	$$ H = \sideset{}{_{n \geq 0}} \bigoplus H_n, $$
where $H_n$ is an irreducible $\gp{K}$-subspace of $H$ generated by a unitary vector $e_n$, such that $\{ e_0, \cdots, e_m \}$ form an orthonormal basis of the $\gp{K}_0[\vp^m]$-invariant subspace of $H$ for any $m \in \ag{N}$. These vectors $e_n$, called ``classical vectors'' \footnote{They are called ``paramodular'' vectors by Brooks Roberts and Ralf Schmidt.} in \cite[Definition 5.4]{Wu14}, are defined up to a factor of modulus $1$. We determine/choose them as follows. First of all, we impose
	$$ e_0(\kappa) = 1, \quad \forall \kappa \in \gp{K}. $$
\begin{lemma}
	Let $e_n' \in \pi_s$ be defined by
	$$ e_1' = \pi_s(a(\varpi^{-1})).e_0 - \frac{q^s+q^{-s}}{q^{1/2}+q^{-1/2}} e_0, $$
	$$ e_n' = \pi_s(a(\varpi^{-n})).e_0 - q^{-1/2} (q^s+q^{-s}) \pi_s(a(\varpi^{-n+1})).e_0 + q^{-1} \pi_s(a(\varpi^{-n+2})).e_0, \quad \forall n \geq 2. $$
	Then if $s \in i\ag{R}$, $\{ e_0, e_1', \dots, e_m' \}$ is an orthogonal basis of the $\gp{K}_0[\vp^m]$-invariant subspace of $H$ for any $m \in \ag{N}$.
\label{BCRep}
\end{lemma}
\begin{proof}
	If $s \in i\ag{R}$, $\{ e_0, \pi_s(a(\varpi^{-1})).e_0, \dots, \pi_s(a(\varpi^{-m})).e_0 \}$ is a basis of the $\gp{K}_0[\vp^m]$-invariant subspace of $H$ for any $m \in \ag{N}$. Then use the Macdonald formula \cite[Theorem 4.6.6]{Bu98} to verify that $e_n'$ is orthogonal to $\pi_s(a(\varpi^{-m})).e_0$ for $0 \leq m \leq n-1$.
\end{proof}
\begin{lemma}
	If we define
	$$ e_1 := \frac{q^{1/2}+q^{-1/2}}{q^{s+1/2}-q^{-s-1/2}} e_1', \quad e_n := \sqrt{\frac{q+1}{q-1}} \frac{q^{-ns}}{1-q^{-1-2s}} e_n', \forall n \geq 2, $$
	then $e_n$ is independent of $s$ and $\{ e_0, \cdots, e_m \}$ form an orthonormal basis of the $\gp{K}_0[\vp^m]$-invariant subspace of $H$ for any $m \in \ag{N}$. Moreover, the dimension $d_n$ of $H_n$ is given by
	$$ d_0=1, d_1=q, d_n=q^n-q^{n-2}, n \geq 2. $$
\label{BCRepN}
\end{lemma}
\begin{proof}
	If we write $ \vo_n = \varpi^n\vo - \varpi^{n+1}\vo, n \geq 1$, then
	$$ D_0 =  \grB(\vo)w\grN(\vo) = \left\{ \begin{pmatrix} a & b \\ c & d \end{pmatrix} \in \grK: c \in \vo^{\times} \right\}, $$
	$$ D_n = \grB(\vo)\grN_-(\vo_n) =  \left\{ \begin{pmatrix} a & b \\ c & d \end{pmatrix} \in \grK: c \in \vo_n \right\}, 1 \leq n \leq m, $$
	$$ D_m' = \grK_0[m] = \cup_{n=m}^{\infty} D_n $$
	are the double cosets of $\grK$ w.r.t. $\grB(\vo)$ and $\grK_0[m]$. From the computation
	$$ \begin{pmatrix} a \varpi^{-n} & b \\ c \varpi^{-n} & d \end{pmatrix} \begin{pmatrix} c^{-1}d \varpi^n & 1 \\ -1 & 0 \end{pmatrix} = \begin{pmatrix} c^{-1}(ad-bc) & * \\ 0 & c \varpi^{-n} \end{pmatrix}, $$
	$$ \begin{pmatrix} a \varpi^{-n} & b \\ c \varpi^{-n} & d \end{pmatrix} \begin{pmatrix} 1 & 0 \\ -d^{-1}c \varpi^{-n} & 1 \end{pmatrix} = \begin{pmatrix} d^{-1} \varpi^{-n} (ad-bc) & * \\ 0 & d \end{pmatrix}, $$
	one easily deduces that
	$$ e_1' \mid_{D_0} = \frac{q^{s+1/2}-q^{-s-1/2}}{q^{1/2}+q^{-1/2}} \cdot (-q^{-1/2}), \quad e_1' \mid_{D_k} = \frac{q^{s+1/2}-q^{-s-1/2}}{q^{1/2}+q^{-1/2}} \cdot q^{1/2}, k \geq 1; $$
	$$ e_n' \mid_{D_k} = 0, 0 \leq k \leq n-2, \quad e_n' \mid_{D_{n-1}} = (1-q^{-1-2s}) q^{n(s+1/2)} (-q^{-1}), $$
	$$ e_n' \mid_{D_k} = (1-q^{-1-2s}) q^{n(s+1/2)} (1-q^{-1}), k \geq n. $$
	The assertion follows since the mass $w_n$ of $D_n$, assuming the mass of $\grK$ is $1$, is given by
	$$ w_0 = \frac{q}{q+1}, w_n=\frac{q^{-(n-1)}}{q+1}(1-q^{-1}), n \geq 1. $$
	For the ``moreover'' part, it suffices to notice
\begin{equation} 
	e_n(\kappa) = d_n^{1/2} \Pairing{\kappa.e_n}{e_n}
\label{CVProj}
\end{equation}
	and to evaluate the above equation at $\kappa=1$.
\end{proof}
\begin{corollary}
	We record some special values of $e_n$:
	$$ e_n(1) = \left\{ \begin{matrix} 1 & n=0; \\ q^{1/2} & n=1; \\ (q^n-q^{n-2})^{1/2} & n \geq 2; \end{matrix} \right. \quad e_n(w) = \left\{ \begin{matrix} 1 & n=0; \\ -q^{-1/2} & n=1; \\ 0 & n \geq 2. \end{matrix} \right. $$
\label{CVSV}
\end{corollary}
\begin{lemma}
	The two base $\{ e_0, e_1, \dots \}$ and $\{ e_0, \pi_s(a(\varpi^{-1})).e_0, \dots \}$ of the subspace of classical vectors in $H$ are related as follows.
\begin{itemize}
	\item[(1)] For $n \geq 2$, we have
	$$ e_1 = \frac{q^{1/2}+q^{-1/2}}{q^{s+1/2}-q^{-s-1/2}} \pi_s(a(\varpi^{-1})).e_0 - \frac{q^s+q^{-s}}{q^{s+1/2}-q^{-s-1/2}} e_0; $$
	$$ e_n = \sqrt{\frac{q+1}{q-1}} \frac{q^{-ns}}{1-q^{-1-2s}} \left( \pi_s(a(\varpi^{-n})).e_0 - q^{-1/2} (q^s+q^{-s}) \pi_s(a(\varpi^{-n+1})).e_0 + q^{-1} \pi_s(a(\varpi^{-n+2})).e_0 \right). $$
	\item[(2)] For $n \geq 2$, we have
	$$ \pi_s(a(\varpi^{-1})).e_0 = \frac{q^{s+1/2}-q^{-s-1/2}}{q^{1/2}+q^{-1/2}} e_1 + \frac{q^s+q^{-s}}{q^{1/2}+q^{-1/2}} e_0; \quad \pi_s(a(\varpi^{-n})).e_0 = \sum_{l=0}^n c(n,l;s) e_l, $$
	where the coefficients $c(n,l;s) = c_{\vp}(n,l;s)$ are given by
	$$ c(n,0;s) = \frac{q^{-\frac{n}{2}}}{1+q^{-1}} \left\{ \frac{q^{(n+1)s}-q^{-(n+1)s}}{q^s-q^{-s}} - q^{-1} \frac{q^{(n-1)s}-q^{-(n-1)s}}{q^s-q^{-s}} \right\}; $$
	$$ c(n,1;s) = \frac{q^{-\frac{n-1}{2}}}{1+q^{-1}} (q^{ns}-q^{-ns}) \frac{1-q^{-1-2s}}{1-q^{-2s}}; $$
	$$ c(n,l;s) = q^{-\frac{n-l}{2}}(q^{ns}-q^{(2l-2-n)s}) \frac{1-q^{-1-2s}}{1-q^{-2s}} \sqrt{\frac{q-1}{q+1}}, 2 \leq l \leq n. $$
	\item[(3)] For $n \geq 2$, we have
	$$ \pi_s(a(\varpi)).e_0 = \frac{q^{s+1/2}-q^{-s-1/2}}{q^{1/2}+q^{-1/2}} w.e_1 + \frac{q^s+q^{-s}}{q^{1/2}+q^{-1/2}} e_0; \quad \pi_s(a(\varpi^n)).e_0 = \sum_{l=0}^n c(n,l;s) w.e_l, $$
	where $w \in \gp{K}$ is the Weyl element and the coefficients $c(n,l;s)$ are the same as in (2).
\end{itemize}
\label{BaseCNA}
\end{lemma}
\begin{proof}
	(1) is merely a re-statement of Lemma \ref{BCRep} and \ref{BCRepN}. For (2), we first use Lemma \ref{BCRep} to deduce a relation of two formal power series
\begin{align*}
	\sum_{n=2}^{\infty} e_n' X^n &= \left( \sum_{n=0}^{\infty} \pi_s(a(\varpi^{-n})).e_0 X^n \right) \left( 1-q^{-1/2}(q^s+q^{-s})X + q^{-1}X^2 \right) \\
	&\quad - e_0 - \pi_s(a(\varpi^{-1})).e_0 X + q^{-1/2}(q^s+q^{-s})e_0X \\
	&= \left( \sum_{n=0}^{\infty} \pi_s(a(\varpi^{-n})).e_0 X^n \right) \left( 1-q^{-1/2}(q^s+q^{-s})X + q^{-1}X^2 \right) \\
	&\quad - e_0 - \left( e_1' - q^{-1} \frac{q^s + q^{-s}}{q^{1/2} + q^{-1/2}} e_0 \right) X.
\end{align*}
	Reverting it, we obtain
\begin{align*}
	\sum_{n=0}^{\infty} \pi_s(a(\varpi^{-n})).e_0 X^n &= \left( \sum_{n=0}^{\infty} q^{-n/2} \frac{q^{(n+1)s} - q^{-(n+1)s}}{q^s - q^{-s}} X^n \right) \\
	&\quad \cdot \left( e_0 + \left( e_1' - q^{-1} \frac{q^s + q^{-s}}{q^{1/2} + q^{-1/2}} e_0 \right) X + \sum_{n=2}^{\infty} e_n' X^n \right)
\end{align*}
	and conclude by inserting Lemma \ref{BCRepN}. (3) follows from (2) by noting
	$$ \pi_s(a(\varpi^n)).e_0 = \pi_s(wa(\varpi^{-n})w^{-1}) \pi_s(\begin{pmatrix} \varpi^n & \\ & \varpi^n \end{pmatrix}).e_0 = w.\pi_s(a(\varpi^{-n})).e_0. $$
\end{proof}
\begin{corollary}
	If $\IntwR(s): \pi_s \to \pi_{-s}$ is the normalized intertwining operator, sending $e_0$ to $e_0$, then $\IntwR(s)$ acts on $H_n$ by multiplication by
	$$ \mu(n;s) = \mu_{\vp}(n,s) = q^{-2ns} \frac{1-q^{-(1-2s)}}{1-q^{-(1+2s)}}. $$
\label{IntwRCal}
\end{corollary}
\begin{proof}
	This is a special case of the computation in \cite[\S 3.4.3]{Wu5}. Here is another proof. $\IntwR(s) e_n$ is equal to
\begin{align*}
	&\quad \sqrt{\frac{q+1}{q-1}} \frac{q^{-ns}}{1-q^{-1-2s}} \left( \pi_{-s}(a(\varpi^{-n})).e_0 - q^{-1/2} (q^s+q^{-s}) \pi_{-s}(a(\varpi^{-n+1})).e_0 + q^{-1} \pi_{-s}(a(\varpi^{-n+2})).e_0 \right) \\
	&= q^{-2ns} \frac{1-q^{-(1-2s)}}{1-q^{-(1+2s)}} \cdot \\
	&\quad \sqrt{\frac{q+1}{q-1}} \frac{q^{ns}}{1-q^{-1+2s}} \left( \pi_{-s}(a(\varpi^{-n})).e_0 - q^{-1/2} (q^s+q^{-s}) \pi_{-s}(a(\varpi^{-n+1})).e_0 + q^{-1} \pi_{-s}(a(\varpi^{-n+2})).e_0 \right),
\end{align*}
	the last line being equal to $e_n$ since it is independent of $s$.
\end{proof}
\begin{remark}
	For $\F$ archimedean, we have similar computations already available in \cite[Proposition 2.6.3]{Bu98} and \cite[\S 2.7]{Wu14}. We recall them without proof.
\begin{itemize}
	\item $\F=\ag{R}$. $H_n$ is the subspace of vectors $v$ such that
	$$ \begin{pmatrix} \cos \alpha & \sin \alpha \\ -\sin \alpha & \cos \alpha \end{pmatrix}.v = e^{in\alpha} v $$
	and $H_n \neq \{ 0 \}$ only if $2 \mid n \in \ag{Z}$. We have
	$$ \mu(n;s) = \mu_v(n,s) = \sideset{}{_{2\mid k = 0}^{\norm[n]-2}} \prod \frac{k+1-2s}{k+1+2s}. $$
	\item $\F=\ag{R}$. $H_n$ is the subspace on which $\SU_2(\ag{C})$ acts as the unitary irreducible representation of dimension $n+1$ and $H_n \neq \{ 0 \}$ only if $2 \mid n \in \ag{N}$. We have
	$$ \mu(n;s) = \mu_v(n,s) = \sideset{}{_{2\mid k = 1}^{n/2}} \prod \frac{k-2s}{k+2s}. $$
\end{itemize}
	We write $e_0 \in H_0$ for the spherical function taking value $1$ on $\gp{K} = \SO_2(\ag{R})$ or $\SU_2(\ag{C})$.
\label{IntwRCalA}
\end{remark}
\begin{definition}
	For $n,l \in \ag{Z}$, we write $l \preceq n$ to mean either $0 \leq l \leq n$ or $n \leq l \leq 0$. We extend the definition of $e_n$ resp. $\mu(n;s)$ resp. $c(n,l;s)$ for $n,l \in \ag{N}$ to $n,l \in \ag{Z}, l \preceq n$ by requiring
	$$ e_{-n} := w.e_n, \quad \mu(-n;s) = \mu(n;s), \quad c(-n,-l;s) = c(n,l;s). $$
\label{IneqExt}
\end{definition}

\section{Local Estimations}

	\subsection{Non Archimedean Places for Exceptional Part}
	
	We work on a non archimedean place $\vp$ and omit the subscript $\vp$ for simplicity of notations. Recall $e_n$ defined in Lemma \ref{BCRepN} and Definition \ref{IneqExt}, but change $s$ to $s_0$. To emphasize the dependence on $s_0$, we write $e_{n,s_0} \in \pi_{s_0}$ for the flat section associated with $e_n$, and $W_n(s_0, \cdot)$ the associated Kirillov function in the Kirillov model $\Kir(\pi_{s_0}, \psi)$ of $\pi_{s_0}$, with respect to an unramified additive character $\psi$ of $\F$. Recall the local zeta functional
	$$ \zeta(s, W) := \int_{\F^{\times}} W(y) \norm[y]^{s-1/2} d^{\times}y, \quad W \in \Kir(\pi_{s_0}, \psi). $$
\begin{lemma}
	The ratios of zeta-functions
	$$ \zeta_{\vp,n}(s,s_0) = \zeta_n(s,s_0) := \frac{\zeta(1/2+s, W_n(s_0,\cdot))}{\zeta(1/2+s, W_0(s_0, \cdot))}, \quad n \in \ag{Z} $$
	are determined by
\begin{itemize}
	\item $\zeta_0(s,s_0) = 1$ and
	$$ \zeta_1(s,s_0) = \frac{q^{1/2} + q^{-1/2}}{q^{s_0+1/2} - q^{-(s_0+1/2)}} q^{-s} - \frac{q^{s_0} + q^{-s_0}}{q^{s_0+1/2} - q^{-(s_0+1/2)}}; $$
	\item if $n \geq 2$, then
	$$ \zeta_n(s,s_0) = \sqrt{\frac{q+1}{q-1}} \frac{q^{-ns_0}}{1-q^{-1-2s_0}} \left( q^{-ns} - q^{-1/2}(q^{s_0} + q^{-s_0}) q^{-(n-1)s} + q^{-1-(n-2)s} \right); $$
	\item if $n < 0$, we have $\zeta_n(s,s_0) = \zeta_{-n}(-s,s_0)$.
\end{itemize}
\label{LocZetaRatio}
\end{lemma}
\begin{proof}
	From the relation of zeta functions
	$$ \zeta(s, a(\varpi^n).W) = \norm[\varpi]^{-ns} \zeta(s,W) = q^{ns} \zeta(s,W), $$
	the desired formulas are simple consequences of those in Lemma \ref{BaseCNA} (1) and (3).
\end{proof}
\begin{corollary}
	Assume $\Re s = \epsilon > 0$ small, $\norm[n]$ is bounded by a constant and $0 \leq k \leq 2$, then we have
	$$ \extnorm{ \frac{\partial^k}{\partial s_0^k} \zeta_n(s,1/2) } \ll_{\epsilon} q^{n\epsilon - \norm[n]/2} (\log q)^k. $$
\label{LocEstNAEx}
\end{corollary}

	\subsection{Non Archimedean Places for $\intL^4$-Norms}
	
	Using the notations of the previous subsection, we define the Rankin-Selberg local zeta ratios for $n_1,n_2,n \in \ag{Z}$ and $s_1,s_2,s \in \ag{C}$
\begin{equation}
	\zeta_{\vp}( \begin{array}{cc|c} n_1 & n_2 & n \\ \overline{s_1} & s_2 & s \end{array} ) = \zeta( \begin{array}{cc|c} n_1 & n_2 & n \\ \overline{s_1} & s_2 & s \end{array} ) := \frac{ \int_{\gp{N}(\F) \backslash \PGL_2(\F)} \overline{W_{n_1}(s_1,g)} W_{n_2}(s_2,g) e_{n,s}(g) dg }{ \int_{\gp{N}(\F) \backslash \PGL_2(\F)} \overline{W_0(s_1,g)} W_0(s_2,g) e_{0,s}(g) dg }.
\label{RSLocZetaRatioDef}
\end{equation}
\begin{lemma}
\begin{itemize}
	\item[(1)] We have
	$$ \zeta( \begin{array}{cc|c} -n_1 & -n_2 & -n \\ \overline{s_1} & s_2 & s \end{array} ) = \zeta( \begin{array}{cc|c} n_1 & n_2 & n \\ \overline{s_1} & s_2 & s \end{array} ). $$
	\item[(2)] Let $n_2=0=s_2$. The ratio is non-vanishing only if $\norm[n_1] = \norm[n]$.
	\item[(3)] Recall the dimension $d_n$ of $H_n$ computed in Lemma \ref{BCRepN}. We have for $n \geq 2$
\begin{align*}
	\zeta( \begin{array}{cc|c} n & 0 & n \\ \overline{0} & 0 & s \end{array} ) &= \sqrt{\frac{q+1}{q-1}} \frac{d_n^{-1/2}}{1-q^{-1}} \cdot \left( q^{-ns} \frac{2+(n-1)(1-q^{-(s+1/2)})}{1+q^{-(s+1/2)}} \right. \\
	&\quad \left. -2q^{-1/2-(n-1)s} \frac{2+(n-2)(1-q^{-(s+1/2)})}{1+q^{-(s+1/2)}} + q^{-1-(n-2)s} \frac{2+(n-3)(1-q^{-(s+1/2)})}{1+q^{-(s+1/2)}} \right),
\end{align*}
	while for $n=1$,
	$$ \zeta( \begin{array}{cc|c} 1 & 0 & 1 \\ \overline{0} & 0 & s \end{array} ) = \frac{q^{1/2}+q^{-1/2}}{q^{1/2}-q^{-1/2}}  \left( \frac{2q^{-s}}{1+q^{-(s+1/2)}} - \frac{2}{q^{1/2}+q^{-1/2}} \right). $$
\end{itemize}
\label{LocNVNAL4_1}
\end{lemma}
\begin{proof}
	(1) is a consequence of the $w$-invariance of the Rankin-Selberg local zeta functional. For (2), we may assume $n \geq 0$ by (1). It suffices to notice that
	$$ \int_{\gp{K}} e_n(\kappa) \overline{W_{n_1}(s_1,g\kappa)} d\kappa $$
	is non-vanishing only if $\norm[n_1] = \norm[n]$, since by (\ref{CVProj}) $\int_{\gp{K}} \overline{e_n(\kappa)} \kappa d\kappa$ is $d_n^{-1/2}$ times the orthogonal projection onto the $e_n$-vector of $H_n$. In particular, we deduce for $n \geq 0$
	$$ \int_{\gp{K}} e_n(\kappa) \overline{W_n(s_1,g\kappa)} d\kappa = d_n^{-1/2} \overline{W_n(s_1,g)}. $$
	Hence for (3), we are reduced to computing
	$$ \int_{\F^{\times}} \overline{W_n(0,a(y))} W_0(0,a(y)) \norm[y]^{s-1/2} d^{\times}y. $$
	By Lemma \ref{BaseCNA} (1), we are again reduced to computing
\begin{align*}
	&\quad \int_{\F^{\times}} \overline{W_0(0,a(y\varpi^{-n}))} W_0(0,a(y)) \norm[y]^{s-1/2} d^{\times}y \\
	&= \sum_{m=0}^{\infty} q^{-m/2}(m+1) \cdot q^{-(n+m)/2}(n+m+1) \cdot q^{-(n+m)(s-1/2)} \\
	&= q^{-ns} \cdot \frac{2+(n-1)(1-q^{-(s+1/2)})}{(1-q^{-(s+1/2)})^3},
\end{align*}
	and conclude from it.
\end{proof}
\begin{corollary}
	For any integer $k \geq 0$, we have
	$$ \extnorm{ \frac{\partial^k}{\partial s^k} \zeta( \begin{array}{cc|c} n & 0 & n \\ \overline{0} & 0 & 1/2 \end{array} ) } \ll_k q^{-\norm[n]} (\log q)^k. $$
\label{LocEstNAL4_1}
\end{corollary}

\section{Proof of Main Result}
	
	The main structure of proof is similar to our former work \cite{Wu14}. We shall only emphasize on the differences and the extra difficulties. We shall not recall the intuition of the method in terms of the equidistribution of certain lines approaching the low lying horocycles, but refer the reader to the first two pages of \cite[\S 3]{Wu14}.

	\subsection{Reduction to Global Period Bound}
	
	The fixed $\GL_2$ automorphic representation $\pi = \pi(1,1)$ is realized as completed Eisenstein series $\eis^*(0,\cdot)$. We imitate the cuspidal case by choosing
	$$ \varphi_0 = \eis^*(0,f_0) \quad \varphi = n(T).\varphi_0, $$
where $f_0$ is the spherical function taking value $1$ on $\gp{K}$ in the induced model of $\pi(1,1)$. Writing the normalized Whittaker functions as
	$$ W_{0,v}^* := \zeta_v(1) W_{0,v} = \zeta_v(1) W_{f_0,v}, $$
	we get by Proposition \ref{EulerZetaEis} an expression of the relevant $L$-function
	$$ L(\frac{1}{2},\chi)^2 = \left[\prod_{v\mid\infty} \int_{\F_v^{\times}} n(T_v).W_{0,v}^*(a(y_v)) d^{\times}y_v \prod_{v<\infty} \frac{\int_{\F_v^{\times}} n(T_v).W_{0,v}^*(a(y_v)) d^{\times}y_v}{L_v(\frac{1}{2},\chi_v)^2} \right]^{-1} \zeta(\frac{1}{2},\chi,\varphi), $$
where the global zeta-integral is defined in Definition \ref{ZetaFunctDef} and reduces in our case to
	$$ \zeta(s,\chi,\varphi) = \int_{\F^{\times} \backslash \ag{A}^{\times}} (\varphi(a(y)) - \varphi_{\grN}(a(y))) \chi(y) \norm[y]_{\ag{A}}^{s-1/2} d^{\times}y, $$
whose value at $s=1/2$ must be interpreted via analytic continuation, unlike the cuspidal case.
\begin{proposition}
	We can choose $T_v$ with $\norm[T_v]_v \in [\Cond(\chi_v), 2 \Cond(\chi_v)]$ such that
	$$ \prod_{v\mid\infty} \int_{\F_v^{\times}} n(T_v).W_{0,v}^*(a(y_v)) d^{\times}y_v \prod_{v<\infty} \frac{\int_{\F_v^{\times}} n(T_v).W_{0,v}^*(a(y_v)) d^{\times}y_v}{L_v(\frac{1}{2},\chi_v)^2} \gg_{\F} Q^{-\frac{1}{2}}, $$
	where $Q=\Cond(\chi)=\Pi_v \Cond(\chi_v)$ is the analytic conductor of $\chi$ (we keep this notation in what follows).
\label{LocMainEst}
\end{proposition}
\begin{proof}
	This is a special case of \cite[Proposition 2.4]{Wu3} for the ``\emph{Option (B)}''.
\end{proof}

	\subsection{Reduction to Bound of Truncated Integral}

	We are reduced to bounding $\zeta(1/2, \chi, \varphi)$. It can be defined only via analytic continuation. However, we can still approximate it with truncated integral, just as what the classical approximate functional equation does. The outcome is that we essentially only need to estimate an integral of compact domain, which is equivalent to a finite sum in the classical setting. Recall Definition \ref{h-TruncZeta}. We shall apply Proposition \ref{ApproxFE} with $A=Q^{-\kappa-1}, B=Q^{\kappa-1}$ for some $\kappa \in (0,1)$ to be chosen later, with $c_1 = c_2 = 1/2+\epsilon$ and with $h_0$ and $h$ specified there.
\begin{lemma}
	Assume $Q$ is bounded away from $0$. Then we have for any small $\epsilon > 0$
\begin{align*}
	\zeta(\frac{1}{2},\chi,\varphi) &= \zeta(\sigma * h, \chi, \varphi) + O_{\F,h_0,\epsilon}(Q^{-\frac{\kappa}{2}+\epsilon}) \\
	&= \int_{\F^{\times} \backslash \A^{\times}} \sigma*h(\norm[y]_{\A}) \varphi(a(y)) \chi(y) d^{\times}y + O_{\F,h_0,\epsilon}(Q^{-\frac{\kappa}{2}+\epsilon}+Q^{\frac{\kappa-1}{2}+\epsilon}),
\end{align*}
	where $\sigma$ is the following average of Dirac measures
	$$ \sigma = \frac{1}{M_E^2} \sum_{v,v' \in I_E} \delta_{\norm[\varpi_v]_v \norm[\varpi_{v'}]_{v'}^{-1}} $$
	with a parameter $E > 0$ to be chosen later and
	$$ I_E = \{ v<\infty: q_v \in [E,2E], T_v=0 \}, \quad M_E = \norm[I_E] \gg \frac{E}{\log E}. $$
\label{TruncEst}
\end{lemma}
\begin{proof}
	We only need to consider the case for $h$ since $\sigma*h$ gives bounded translations. Compared to \cite[Lemma 3.2]{Wu14}, the new situation is:
\begin{itemize}
	\item The normal polar part is non-vanishing;
	\item At $v \mid \infty$, $W_{0,v}^*$ is no longer of compact support in $\F_v^{\times}$, hence \cite[Corollary 4.3]{Wu14} used in \cite[Section 6.1]{Wu14} for local archimedean bounds on the vertical line $\Re s = -1/2-\epsilon$ need to be re-considered;
	\item There is a new passage from the first line to the second line, for which the estimation of an integral against the constant term $\varphi_{\gp{N}}$ need to be done.
\end{itemize}
	We proceed to bound each part appearing in Proposition \ref{ApproxFE} one by one.
	
\noindent (0) The degenerate polar part is vanishing.
	
\noindent (1) The normal polar part and the integral against $\varphi_{\grN}$ are non-vanishing only if $\chi=\norm^{i\mu}$ for some $\mu \in \R$, in which case $\Cond(\chi_v)=1$ for all $v<\infty$ and $\Cond(\chi_v) \asymp \norm[\mu]^{[\F_v:\R]}$ for all $v \mid \infty$. Hence $Q \asymp \norm[\mu]^{[\F:\Q]}$. Note that by \cite[Proposition 7.33]{Wu9}, there are $\mu_1,\mu_2 \in \C$ depending only on $\F$ such that
\begin{equation}
	\varphi_{0,\grN}(zn(x)a(y)\kappa) = \mu_1 \norm[y]_{\A}^{\frac{1}{2}} + \mu_2 \norm[y]_{\A}^{\frac{1}{2}} \log \norm[y]_{\A}, \quad \forall z \in \grZ(\A), x \in \A, y \in \A^{\times}, \kappa \in \grK.
\label{TestCst}
\end{equation}
	If we write
	$$ \varphi_{\grN}(zn(x)a(y)\kappa) = \norm[y]_{\A}^{\frac{1}{2}} f_1(\kappa) + \norm[y]_{\A}^{\frac{1}{2}} \log \norm[y]_{\A} f_2(\kappa), $$
	then we can easily calculate
	$$ f_1(1) = \mu_1, \quad f_2(1) = \mu_2, \quad f_1(w) = \mu_1 \sideset{}{_{v \mid \infty}} \prod (1+\norm[T_v]^2)^{-\frac{[\F_v:\R]}{2}}, $$
	$$ f_2(w) = \mu_2 \sideset{}{_{v \mid \infty}} \prod (1+\norm[T_v]^2)^{-\frac{[\F_v:\R]}{2}} \log \left( \sideset{}{_{v \mid \infty}} \prod (1+\norm[T_v]^2)^{-\frac{[\F_v:\R]}{2}} \right). $$
	We thus find that the normal polar part, which is of the form
	$$ O\left( \frac{A^{1/2} \norm[f_1(1)]}{\norm[1/2+i\mu]^N} + \frac{A^{1/2} \norm[f_2(1) \log A]}{\norm[1/2+i\mu]^N} \right) + O\left( \frac{A^{-1/2} \norm[f_1(w)]}{\norm[1/2-i\mu]^N} + \frac{A^{-1/2} \norm[f_2(w) \log A]}{\norm[1/2-i\mu]^N} \right) $$
	can be bounded as $O(Q^{-N})$ for any $N \in \N$, due to the arbitrarily large denominators. For the integral against $\varphi_{\grN}$, since $h(t)$ has support contained in $[Q^{-\kappa-1},Q^{\kappa-1}]$ with $\norm[h(t)] \leq 1$, we find that (note that $n(T).\varphi_{\gp{N}}(a(y)) = \varphi_{\gp{N}}(a(y))$)
\begin{equation}
	\int_{\F^{\times} \backslash \A^{\times}} h(\norm[y]_{\A}) \varphi_{\grN}(a(y) \chi(y) d^{\times}y = O_{\F,h_0,\epsilon}(Q^{\frac{\kappa-1}{2}+\epsilon}).
\label{InthEst}
\end{equation}
	
\noindent (2) We turn to the lower part. Recall the choice $c_1=1/2+\epsilon$. The relevant zeta-function has a decomposition as a finite product
\begin{align*}
	\zeta(\frac{1}{2}+s,\chi^{-1},w.\varphi) &= L(\frac{1}{2}+s,\chi^{-1})^2 \cdot \prod_{v \mid \infty} \int_{\F_v^{\times}} wn(T_v).W_{0,v}^*(a(y_v)) \chi_v^{-1}(y_v) \norm[y_v]_v^{s} d^{\times} y_v \\
	&\quad \cdot \prod_{v < \infty} \frac{\int_{\F_v^{\times}} wn(T_v).W_{0,v}^*(a(y_v)) \chi_v^{-1}(y_v) \norm[y_v]_v^{s} d^{\times} y_v}{L_v(\frac{1}{2}+s,\chi_v^{-1})^2}.
\end{align*}
	At an archimedean place $v$, say $\F_v = \R$, to the local integral
\begin{align*}
	&\quad \int_{\R^{\times}} wn(T_v).W_{0,v}^*(a(y_v)) \chi_v^{-1}(y_v) \norm[y_v]_v^{s} d^{\times} y_v \\
	&= \chi_v(1+T_v^2)^{-1} (1+T_v^2)^{s} \int_{\R^{\times}} W_{0,v}^*(a(y)) \psi_v(yT_v) \chi_v^{-1}(y_v) \norm[y_v]_v^{s} d^{\times}y_v
\end{align*}
	\cite[Lemma 3.12 (2)]{Wu3} gives a bound as $O(\norm[T_v]_v^{1/2+\epsilon})$. For $\F_v=\C$ the argument is similar, using \cite[Lemma 3.13 (2)]{Wu3}. At $v<\infty$, \cite[Corollary 4.8]{Wu14} is still applicable. We thus deduce that, using the convex bound of $L(1/2+s,\chi^{-1})$,
	$$ \extnorm{\zeta(\frac{1}{2}+s,\chi^{-1},w.\varphi)} \ll_{\F,\epsilon} \norm[\frac{1}{2}+s]^{\epsilon} \Cond(\chi)^{\frac{1}{2}+\epsilon}. $$
	The desired bound is thus $O(Q^{-\frac{\kappa +1}{2}}Q^{\frac{1}{2}+\epsilon}) = O(Q^{-\frac{\kappa}{2}+\epsilon})$.
	
\noindent (3) The treatment of the upper part is similar and simpler. It gives the desired bound $O(Q^{-\frac{\kappa }{2}+\epsilon})$.
\end{proof}

	\subsection{Interlude: Failure of Truncation on Eisenstein series}
	\label{FailTE}
	
	We have approximated $\zeta(1/2,\chi,\varphi)$ by some smoothly truncated integral
	$$ \int_{\F^{\times} \backslash \ag{A}^{\times}}^h \varphi(a(y)) \chi(y) d^{\times}y \quad \left( \text{where } \int_{\F^{\times} \backslash \ag{A}^{\times}}^h := \int_{\F^{\times} \backslash \ag{A}^{\times}} h(\norm) \right) $$
as in the cuspidal case. We then would like to apply the Cauchy-Schwarz inequality
	$$ \extnorm{\int_{\F^{\times} \backslash \ag{A}^{\times}}^h \varphi(a(y)) \chi(y) d^{\times}y}^2 \leq \int_{\F^{\times} \backslash \ag{A}^{\times}}^h d^{\times}y \cdot \int_{\F^{\times} \backslash \ag{A}^{\times}}^h n(T).\extnorm{\varphi_0(a(y))}^2 d^{\times}y, $$
and apply Fourier inversion to $\norm[\varphi_0]^2$, interchange the order of summation and estimate each component as in the cuspidal case. This is not possible because $\varphi_0$ is no longer of rapid decay hence $\norm[\varphi_0]^2$ is not square integrable any more. A first idea, which was already employed in \cite[\S 5.1.7]{MV10}, is to (smoothly) truncate the Eisenstein series $\varphi_0$ up to some height $X$, denoted by $\Lambda^X \varphi_0$. For example, if we naively choose $X$ no less than the height of the truncation on the integral (namely, $Q^{\kappa +1}$ in the notation of \cite{Wu14}), we find
	$$ \int_{\F^{\times} \backslash \ag{A}^{\times}}^h \varphi(a(y)) \chi(y) d^{\times}y = \int_{\F^{\times} \backslash \ag{A}^{\times}}^h n(T). \Lambda^X \varphi_0(a(y)) \chi(y) d^{\times}y. $$
We could continue the argument by replacing $\varphi_0$ with $\Lambda^X \varphi_0$. But then some $\intL^2$-Sobolev norm of $\extnorm{\Lambda^X \varphi_0}^2$ come in as a multiplicative factor of the final estimation of the above integral. This causes no problem in the cuspidal case since the relevant norm is bounded by some $\intL^4$-Sobolev norms of $\varphi_0$, which depends only on $\pi$. This is no longer the case for $\Lambda^X \varphi_0$ since its norms all depend on $X$, hence some positive power of $Q=\Cond(\chi)$. This means that in the final optimization just before \cite[Remark 3.11]{Wu14}, we would have to replace $EQ^{-1/4+\theta/2}$ by something like $XEQ^{-1/4+\theta/2}$, which completely destroys the Burgess-like quality. Indeed in the thesis version of \cite{Wu14} we have pursued this idea and were only able to obtain a saving $(1-2\theta)/12$ instead of the Burgess-like saving $(1-2\theta)/8$.

\noindent A better way, which is also the main innovation of this paper, is to generalize the spectral decomposition/Fourier inversion into a space of functions suitably larger than the square-integrable ones. As we have seen in \S \ref{RTPF}, $\Aut^{\freg}(\GL_2,1)$ is a good candidate: it contains $\norm[\varphi_0]^2$ and differs from smooth vectors in the $\intL^2$-space only by (non-unitary) Eisenstein series. Precisely, we shall decompose
	$$ \int_{\F^{\times} \backslash \ag{A}^{\times}}^h n(T).\extnorm{\varphi_0(a(y))}^2 d^{\times}y = \int_{\F^{\times} \backslash \ag{A}^{\times}}^h n(T).\left( \extnorm{\varphi_0}^2 - \Reis \right)(a(y)) d^{\times}y + \int_{\F^{\times} \backslash \ag{A}^{\times}}^h n(T).\Reis(a(y)) d^{\times}y, $$
where we have written $\Reis = \Reis(\extnorm{\varphi_0}^2)$ for simplicity. Without amplification, the norms of $\extnorm{\varphi_0}^2 - \Reis$ depend only on $\varphi_0$ hence $\pi$; with amplification the relevant norms have contributions as small as $(\log E)^3$ (see Theorem \ref{L4Bd}) where $E$ denotes the length of the amplifiers, which is negligible. Hence we can treat the term related with $\extnorm{\varphi_0}^2 - \Reis$ in the same way as in the cuspidal case without harming the quality of the bound. Since $\Reis$ is determined explicitly by $\varphi_0$, the term related with it is explicitly estimable. We will treat the estimation and see that its contribution does not harm the quality of the final bound, neither. Namely, the generalized spectral decomposition fits as well with the estimation of the integrals as the ordinary one in the cuspidal case. Note that the simpler $\varphi_0$ is, the simpler $\Reis$ is.

	\subsection{Regroupment of Generalized Fourier Inversion}

	We make the strategy described in the above remark more precise. Introducing
	$$ \sigma_{\chi}' = \frac{1}{M_E^2} \sum_{\vp,\vp' \in I_E} \chi \left( \frac{\varpi_{\vp}}{\varpi_{\vp'}} \right) \delta_{\varpi_{\vp} \varpi_{\vp'}^{-1}}, $$
	we can apply the Cauchy-Schwarz inequality
\begin{align*}
	\extnorm{\int_{\F^{\times} \backslash \A^{\times}} \sigma*h(\norm[y]_{\A}) \varphi(a(y)) \chi(y) d^{\times}y}^2 &= \extnorm{\int_{\F^{\times} \backslash \A^{\times}} h(\norm[y]_{\A}) \sigma_{\chi}'*\varphi(a(y)) \chi(y) d^{\times}y}^2 \\
	&\leq \int_{\F^{\times} \backslash \A^{\times}} h(\norm[y]_{\A}) d^{\times}y \cdot \int_{\F^{\times} \backslash \A^{\times}} h(\norm[y]_{\A}) \norm[\sigma_{\chi}' * \varphi(a(y))]^2 d^{\times}y.
\end{align*}
	The first integral in the last line is of size $O_{\F}(\log Q)$, hence negligible. Opening the square, we get
\begin{align*}
	\norm[\sigma_{\chi}' * \varphi(a(y))]^2 &= \frac{1}{M_E^4} \sum_{\vp_1,\vp_1',\vp_2,\vp_2' \in I_E} \chi \left( \frac{\varpi_{\vp_1}}{\varpi_{\vp_1'}} \right) \chi^{-1} \left( \frac{\varpi_{\vp_2}}{\varpi_{\vp_2'}} \right) \left( a\left( \frac{\varpi_{\vp_1}}{\varpi_{\vp_1'}} \right).\varphi \cdot \overline{a\left( \frac{\varpi_{\vp_2}}{\varpi_{\vp_2'}} \right).\varphi} \right)(a(y)) \\
	&= \frac{1}{M_E^4} \sum_{\vec{\vp} \in I_E^4} \chi_{\vec{\vp}} a\left( \frac{\varpi_{\vp_2}}{\varpi_{\vp_2'}} \right)n(T). \left( a\left( \frac{\varpi_{\vp_1}\varpi_{\vp_2'}}{\varpi_{\vp_1'}\varpi_{\vp_2}} \right).\varphi_0 \cdot \overline{\varphi_0} \right)(a(y)),
\end{align*}
	where we have abbreviated
	$$ \chi_{\vec{\vp}} := \chi \left( \frac{\varpi_{\vp_1}}{\varpi_{\vp_1'}} \right) \chi^{-1} \left( \frac{\varpi_{\vp_2}}{\varpi_{\vp_2'}} \right) \in \ag{C}^{(1)}. $$
	Decomposing the non square-integrable function as
\begin{equation}
	a\left( \frac{\varpi_{\vp_1}\varpi_{\vp_2'}}{\varpi_{\vp_1'}\varpi_{\vp_2}} \right).\varphi_0 \cdot \overline{\varphi_0} = a\left( \frac{\varpi_{\vp_2'}}{\varpi_{\vp_2}} \right).\varphi_0(\vec{\vp}) + \Reis_0(\vec{\vp}),
\label{RegPhi0SqDef}
\end{equation}
	where the $\intL^2$-residual part (\ref{ReisDef}) or \cite[Definition 5.26]{Wu9} is given an abbreviated notation
\begin{equation}
	\Reis_0(\vec{\vp}) := \Reis \left( a\left( \frac{\varpi_{\vp_1}\varpi_{\vp_2'}}{\varpi_{\vp_1'}\varpi_{\vp_2}} \right).\varphi_0 \cdot \overline{\varphi_0} \right),
\label{ResPhi0SqDef}
\end{equation}
	applying to $\varphi_0(\vec{\vp}) = \varphi_0(\vec{\vp})_{\grN} + \varphi_0(\vec{\vp})_{\mathrm{cusp}} + \varphi_0(\vec{\vp})_{\mathrm{Eis}}$ the Fourier inversion decomposition in the sense of \cite[Theorem 2.18]{Wu14} and regrouping the two constant terms, we can rewrite the second integral as
\begin{align*}
	&\quad \int_{\F^{\times} \backslash \A^{\times}} h(\norm[y]_{\A}) \norm[\sigma_{\chi}' * \varphi(a(y))]^2 d^{\times}y \\
	&= \frac{1}{M_E^4} \sum_{\vec{\vp} \in I_E^4} \chi_{\vec{\vp}} \int_{\F^{\times} \backslash \A^{\times}} h(\norm[y]_{\A}) \left( a\left( \frac{\varpi_{\vp_1}}{\varpi_{\vp_1'}} \right).\varphi_0 \cdot \overline{a\left( \frac{\varpi_{\vp_2}}{\varpi_{\vp_2'}} \right).\varphi_0} \right)_{\grN} (a(y)) d^{\times}y \\
	&+ \frac{1}{M_E^4} \sum_{\vec{\vp} \in I_E^4} \chi_{\vec{\vp}} \zeta(h,1, n(T).\varphi_0(\vec{\vp})_{\mathrm{cusp}} ) + \frac{1}{M_E^4} \sum_{\vec{\vp} \in I_E^4} \chi_{\vec{\vp}} \zeta(h,1, n(T).\varphi_0(\vec{\vp})_{\mathrm{Eis}} ) \\
	&+ \frac{1}{M_E^4} \sum_{\vec{\vp} \in I_E^4} \chi_{\vec{\vp}} \zeta(h_{\extnorm{ \varpi_{\vp_2'} / \varpi_{\vp_2} }_{\ag{A}}},1, n(T).\Reis_0(\vec{\vp})),
\end{align*}
	where we have used the $h$-truncated zeta-integral in Definition \ref{h-TruncZeta}. Note that we can drop $\varpi_{\vp_2'} / \varpi_{\vp_2}$ in the second integrand, since its adelic norm is contained in $[1/2,2]$.

	\subsection{Bounds for Each Part}
	
\begin{lemma}
	We have for any $\epsilon >0$
	$$ \extnorm{\frac{1}{M_E^4} \sum_{\vec{\vp} \in I_E^4} \chi_{\vec{\vp}} \int_{\F^{\times} \backslash \A^{\times}} h(\norm[y]_{\A}) \left( a\left( \frac{\varpi_{\vp_1}}{\varpi_{\vp_1'}} \right).\varphi_0 \cdot \overline{a\left( \frac{\varpi_{\vp_2}}{\varpi_{\vp_2'}} \right).\varphi_0} \right)_{\grN} (a(y)) d^{\times}y} \ll_{\F,\epsilon} E^{-2+\epsilon} Q^{\epsilon} + Q^{\kappa-1+\epsilon}. $$
\end{lemma}
\begin{proof}
	We write and decompose
	$$ S_{\grN}(\vec{\vp};h) := \int_{\F^{\times} \backslash \A^{\times}} h(\norm[y]_{\A}) \left( a\left( \frac{\varpi_{\vp_1}}{\varpi_{\vp_1'}} \right).\varphi_0 \cdot \overline{a\left( \frac{\varpi_{\vp_2}}{\varpi_{\vp_2'}} \right).\varphi_0} \right)_{\grN} (a(y)) d^{\times}y, $$
	$$ S_{\grN}(\vec{\vp};h) = S_{\grN}^{\Whi}(\vec{\vp};h) + S_{\grN}^*(\vec{\vp};h), $$
	$$ S_{\grN}^*(\vec{\vp};h) = \int_{\F^{\times} \backslash \A^{\times}} h(\norm[y]_{\A}) a\left( \frac{\varpi_{\vp_1}}{\varpi_{\vp_1'}} \right).\varphi_{\grN}(a(y)) \cdot \overline{a\left( \frac{\varpi_{\vp_2}}{\varpi_{\vp_2'}} \right).\varphi_{\grN}(a(y))} d^{\times}y. $$
	The treatment of
	$$ \frac{1}{M_E^4} \sum_{\vec{\vp} \in I_E^4} \chi_{\vec{\vp}} S_{\grN}^{\Whi}(\vec{\vp};h) $$
	is the same as \cite[Lemma 3.4]{Wu14}, which gives a term $\ll_{\F,\epsilon} E^{-2+\epsilon} Q^{\epsilon}$. Using (\ref{TestCst}), we find
\begin{align*}
	S_{\grN}^*(\vec{\vp};h) &= \left(\norm[\mu_1]^2 + \mu_1\overline{\mu_2} \log \extnorm{\frac{\varpi_{\vp_2}}{\varpi_{\vp_2'}}}_{\A} + \mu_2\overline{\mu_1} \log \extnorm{\frac{\varpi_{\vp_1}}{\varpi_{\vp_1'}}}_{\A} \right) \extnorm{\frac{\varpi_{\vp_1}}{\varpi_{\vp_1'}}}_{\A}^{\frac{1}{2}} \extnorm{\frac{\varpi_{\vp_2}}{\varpi_{\vp_2'}}}_{\A}^{\frac{1}{2}} \\
	&\quad \cdot \int_{\F^{\times} \backslash \A^{\times}} h(\norm[y]_{\A}) \norm[y]_{\A} d^{\times}y \\
	&+ \left(\mu_1\overline{\mu_2} + \mu_2\overline{\mu_1} \right) \extnorm{\frac{\varpi_{\vp_1}}{\varpi_{\vp_1'}}}_{\A}^{\frac{1}{2}} \extnorm{\frac{\varpi_{\vp_2}}{\varpi_{\vp_2'}}}_{\A}^{\frac{1}{2}} \cdot \int_{\F^{\times} \backslash \A^{\times}} h(\norm[y]_{\A}) \norm[y]_{\A} \log \norm[y]_{\A} d^{\times}y \\
	&+ \norm[\mu_2]^2 \extnorm{\frac{\varpi_{\vp_1}}{\varpi_{\vp_1'}}}_{\A}^{\frac{1}{2}} \extnorm{\frac{\varpi_{\vp_2}}{\varpi_{\vp_2'}}}_{\A}^{\frac{1}{2}} \cdot \int_{\F^{\times} \backslash \A^{\times}} h(\norm[y]_{\A}) \norm[y]_{\A} \log^2 \norm[y]_{\A} d^{\times}y,
\end{align*}
	from which we easily see, by the same consideration of (\ref{InthEst}),
	$$ \extnorm{S_{\grN}^*(\vec{\vp};h)} \ll_{\F,\epsilon} Q^{\kappa-1+\epsilon}; \quad \extnorm{\frac{1}{M_E^4} \sum_{\vec{\vp} \in I_E^4} \chi_{\vec{\vp}} S_{\grN}^*(\vec{\vp};h)} \ll_{\F,\epsilon} Q^{\kappa-1+\epsilon}. $$
\end{proof}
\begin{remark}
	There are nine different patterns of the positions of $\vp_1,\vp_1',\vp_2,\vp_2'$, listed in \cite[Proposition 3.5]{Wu14}. The estimation of the rest three terms depends on the pattern. For simplicity, we only treat the \textit{typical} pattern in detail in what follows, i.e., when $\vp_1,\vp_1',\vp_2,\vp_2'$ are distinct. The treatment of the other patterns is quite similar.
\end{remark}
\begin{lemma}
	For any $\epsilon > 0$ we have
	$$ \extnorm{\frac{1}{M_E^4} \sum_{\vec{\vp} \in I_E^4} \chi_{\vp} \zeta(h,1, n(T).\Reis_0(\vec{\vp}))} \ll_{\F,h_0,\epsilon} (EQ)^{\epsilon}(Q^{\kappa-1} E^2 + E^{-2}). $$
\end{lemma}
\begin{proof}
	This follows from Corollary \ref{ReisZetaEst} and the omitted calculation for other patterns. The situation is quite similar to that of \cite[Section 6.2 - 6.4]{Wu14}.
\end{proof}

	The estimation of $\frac{1}{M_E^4} \sum_{\vec{\vp} \in I_E^4} \chi_{\vec{\vp}} \zeta(h,1, n(T).\varphi_0(\vec{\vp})_{\mathrm{cusp}} )$ and $\frac{1}{M_E^4} \sum_{\vec{\vp} \in I_E^4} \chi_{\vec{\vp}} \zeta(h,1, n(T).\varphi_0(\vec{\vp})_{\mathrm{Eis}} )$ is essentially the same as the cuspidal case given in \cite[Section 6.3 \& 6.4]{Wu14}. In fact, the only difference appears in \cite[(6.16)]{Wu14}, where we could bound $\Norm[\Delta_{\infty}^{A'}a\left( \frac{\varpi_{v_1}}{\varpi_{v_1'}} \right).\varphi_0 \cdot \overline{a\left( \frac{\varpi_{v_2}}{\varpi_{v_2'}} \right).\varphi_0}]$ easily by $\intL^4$-norm of $\varphi_0$. For the current case, we need to bound $\Norm[\Delta_{\infty}^{A'}\varphi_0(v_1,v_1',v_2,v_2')]$ defined in (\ref{RegPhi0SqDef}). Decomposing the relevant function into $\grK_{\infty}$-isotypic parts, we can apply Theorem \ref{L4Bd}. Thus unlike the cuspidal case, we get an extra $(\log E)^3$ into our estimation, which is harmless. Hence \cite[Lemma 3.6 \& 3.7]{Wu14} remain valid in the current case, giving
\begin{lemma}
	For any $\epsilon > 0$ we have
	$$ \extnorm{\frac{1}{M_E^4} \sum_{\vec{\vp} \in I_E^4} \chi_{\vec{\vp}} \zeta(h,1, n(T).\varphi_0(\vec{\vp})_{\mathrm{cusp}} ) } \ll_{\F,h_0,\epsilon} (EQ)^{\epsilon} E^2 Q^{1/2-\theta}. $$
\end{lemma}
\begin{lemma}
	For any $\epsilon > 0$ we have
	$$ \extnorm{\frac{1}{M_E^4} \sum_{\vec{\vp} \in I_E^4} \chi_{\vec{\vp}} \zeta(h,1, n(T).\varphi_0(\vec{\vp})_{\mathrm{Eis}} ) } \ll_{\F,h_0,\epsilon} (EQ)^{\epsilon} E Q^{(\kappa -1)/2}. $$
\end{lemma}
	
	We are finally lead to establishing (\ref{MainBd}) by
	$$
	\min_{\kappa, E} \max (E^{-1}, EQ^{-1/4+\theta /2}, Q^{-\kappa/2}, Q^{(\kappa-1)/2}, E^{1/2}Q^{(\kappa-1)/4}, EQ^{(\kappa-1)/2}) = Q^{-\frac{1-2\theta}{8}},
	$$
	with an optimal choice given by $$E = Q^{\frac{1-2\theta}{8}}, \kappa = \frac{1}{4}+\frac{\theta}{6}. $$

\section{Complements of Global Estimations}

	\subsection{Estimation for Exceptional Part}
	
	Recall $\Reis_0(\vec{\vp})$ defined in (\ref{ResPhi0SqDef}).
\begin{definition}
	For $\vec{\vp}$, we define $\vec{n}(\vec{\vp})=(n_v)_v$ for $v$ running over the set of places of $\F$ such that
\begin{itemize}
	\item $n_v=0$ for $v \mid \infty$ and $v \notin \{ \vp_1, \vp_1', \vp_2, \vp_2' \}$;
	\item $n_{\vp} = 1$ if $\vp \in \{ \vp_1', \vp_2 \}$, $n_{\vp} = -1$ if $\vp \in \{ \vp_1, \vp_2' \}$.
\end{itemize}
	For $\vec{n}=(n_v)_v, \vec{l} = (l_v)_v$ with components in $\ag{Z}$, we write $\vec{l} \preceq \vec{n}$ to mean $l_v \preceq n_v$ at each $v$, defined in Definition \ref{IneqExt}. We define for $\vec{l} \preceq \vec{n}$
	$$ \sigma(\vec{l}) = \sideset{}{_v} \sum l_v, \quad \Norm[\vec{l}] = \sideset{}{_v} \sum \norm[l_v], \quad e_{\vec{n}} = \otimes_v' e_{n_v},$$
	$$ \mu(\vec{n};s) = \sideset{}{_v} \prod \mu_v(n_v;s), \quad c(\vec{n}, \vec{l}; s) := \sideset{}{_{\vp < \infty}} \prod c_{\vp}(n_{\vp}, l_{\vp}; s) $$
	where the local components are defined in Lemma \ref{BaseCNA} (2), Corollary \ref{IntwRCal}, Remark \ref{IntwRCalA} and Definition \ref{IneqExt}. We write the Laurent expansion at $s=1$ of the complete zeta function $\Lambda_{\F}(s)$ as
\begin{equation}
	\Lambda_{\F}(s) = \frac{1}{s-1} \Lambda_{\F}^* + \gamma_{\F} + O((s-1)).
\label{DedkindZExp}
\end{equation}
	We recall $\eis^{\reg}(s,f)$ in (\ref{RegEisDef}) or \cite[Definition 5.16]{Wu9} as well as the abbreviation
	$$ \eis^{\reg, (n)}(s,f) := \frac{\partial^n}{\partial s^n} \eis^{\reg}(s,f). $$
\label{SpNot}
\end{definition}
\noindent We compute $\Reis_0(\vec{\vp})$ explicitly as
\begin{align*}
	\Reis_0(\vec{\vp}) &= \sum_{\vec{l} \preceq \vec{n}(\vec{\vp})} c(\vec{n}(\vec{\vp}),\vec{l};0) \left\{ \norm[\Lambda_{\F}^*]^2 \eis^{\reg,(2)}(\frac{1}{2},e_{\vec{l}}) + (2\gamma_{\F} - \frac{1}{2} \Lambda_{\F}^* \mu'(\vec{l}; 0) ) \overline{2\gamma_{\F}} \eis^{\reg}(\frac{1}{2},e_{\vec{l}}) \right. \\
	&\quad \left. + \left( 2 \Lambda_{\F}^* \overline{\gamma_{\F}} + \overline{\Lambda_{\F}^*} (2\gamma_{\F} - \frac{1}{2} \Lambda_{\F}^* \mu'(\vec{l}; 0)) \right) \eis^{\reg,(1)}(\frac{1}{2},e_{\vec{l}}) \right\},
\end{align*}
	where the derivative $\mu'(\vec{l};0)$ is taken with respect to $s$ in $\mu(\vec{l};s)$. Consequently, we obtain
\begin{align}
	\zeta(h,1, n(T).\Reis_0(\vec{\vp})) &= \sum_{\vec{l} \preceq \vec{n}(\vec{\vp})} c(\vec{n}(\vec{\vp}),\vec{l};0) \left\{ \norm[\Lambda_{\F}^*]^2 \zeta(h,1,n(T).\eis^{\reg,(2)}(\frac{1}{2},e_{\vec{l}})) \right. \label{SReisSpecDecomp} \\
	&\quad + (2\gamma_{\F} - \frac{1}{2} \Lambda_{\F}^* \mu'(\vec{l}; 0) ) \overline{2\gamma_{\F}} \zeta(h,1,n(T).\eis^{\reg}(\frac{1}{2},e_{\vec{l}})) \nonumber \\
	&\quad \left. + \left( 2 \Lambda_{\F}^* \overline{\gamma_{\F}} + \overline{\Lambda_{\F}^*} (2\gamma_{\F} - \frac{1}{2} \Lambda_{\F}^* \mu'(\vec{l}; 0)) \right) \zeta(h,1,n(T).\eis^{\reg,(1)}(\frac{1}{2},e_{\vec{l}})) \right\}. \nonumber
\end{align}
	Thus we are reduced to bounding, for $n=0,1,2$
	$$ \zeta(h,1,n(T).\eis^{\reg,(n)}(\frac{1}{2},e_{\vec{l}})). $$
\begin{lemma}
	For $n=0,1,2$ and any $\epsilon>0$ sufficiently small, we have
	$$ \extnorm{\zeta(h,1,n(T).\eis^{\reg,(n)}(\frac{1}{2},e_{\vec{l}})} \ll_{\F,h_0,\epsilon} (EQ)^{\epsilon}(Q^{\kappa-1} E^{-\frac{1}{2} \sigma(\vec{l})} + E^{-\frac{1}{2} \Norm[\vec{l}]}). $$
\end{lemma}
\begin{proof}
	Since the Mellin transform the integrand is explicitly related to $L$-functions, we depart from Mellin inversion as
	$$ \zeta(h,1,n(T).\eis^{\reg,(n)}(\frac{1}{2},e_{\vec{l}})) = \int_{\Re s \gg 1} \Mellin{h}(-s) \zeta(1/2+s, 1, n(T).\eis^{\reg,(n)}(\frac{1}{2},e_{\vec{l}})) \frac{ds}{2\pi i}, $$
	and then shift the vertical line of integration to the left. There are poles of the integrand determined by Proposition \ref{GlobZetaFR}. We calculate the constant terms in order to analyze the poles. We have
	$$ \eisCst^{\reg}(\frac{1}{2}+s,e_{\vec{l}})(a(y)k) = \norm[y]_{\A}^{1+s} e_{\vec{l}}(k) + \frac{2s\Lambda_{\F}(-2s)}{\Lambda_{\F}(2+2s)} \frac{\mu(\vec{l};\frac{1}{2}+s)}{2s} \norm[y]_{\A}^{-s} e_{\vec{l}}(k), \vec{l} \neq 0; $$
	$$ \eisCst^{\reg}(\frac{1}{2}+s)(a(y)k,e_{\vec{0}}) = \norm[y]_{\A}^{1+s} e_{\vec{0}}(1) + \frac{2s\Lambda_{\F}(-2s)}{\Lambda_{\F}(2+2s)} \frac{\norm[y]_{\A}^{-s}-1}{2s} e_{\vec{0}}(1). $$
	Hence we get for $n=0,1,2$, $\vec{l} \neq \vec{0}$ and with constants $c_k$ depending only on $\F$
	$$ n(T).\eisCst^{\reg,(n)}(\frac{1}{2},e_{\vec{l}})(a(y)) = \norm[y]_{\A} \log^n \norm[y]_{\A} e_{\vec{l}}(1) + \sum_{k=\Norm[\vec{l}]-1}^n c_k E^{-\Norm[\vec{l}]} (\log E)^{k-\Norm[\vec{l}]+1} \log^{n-k} \norm[y]_{\A} e_{\vec{l}}(1), $$
\begin{align*}
	n(T).\eisCst^{\reg,(n)}(\frac{1}{2},e_{\vec{l}})(a(y)w) &= (\Ht(wn(T))\norm[y]_{\A}) \log^n (\Ht(wn(T)) \norm[y]_{\A}) e_{\vec{l}}(w) + \\
	&\quad \sum_{k=\Norm[\vec{l}]-1}^n c_k E^{-\Norm[\vec{l}]} (\log E)^{k-\Norm[\vec{l}]+1} \log^{n-k} (\norm[y]_{\A}\Ht(wn(T))) e_{\vec{l}}(w);
\end{align*}
	while for $\vec{l}=\vec{0}$
	$$ n(T).\eisCst^{\reg,(n)}(\frac{1}{2},e_{\vec{0}})(a(y)) = \norm[y]_{\A} \log^n \norm[y]_{\A} e_{\vec{0}}(1) + \sum_{k=0}^n c_k \log^{n-k+1} \norm[y]_{\A} e_{\vec{0}}(1), $$
\begin{align*}
	n(T).\eisCst^{\reg,(n)}(\frac{1}{2},e_{\vec{0}})(a(y)w) &= (\Ht(wn(T))\norm[y]_{\A}) (\log \Ht(wn(T)) \norm[y]_{\A})^n e_{\vec{0}}(1) + \\
	&\quad \sum_{k=0}^n c_k (\log \norm[y]_{\A}\Ht(wn(T)))^{n-k+1} e_{\vec{0}}(1).
\end{align*}
	By Proposition \ref{GlobZetaFR}, $\zeta(\frac{1}{2}+s,1,n(T).\eis^{\reg,(n)}(\frac{1}{2},e_{\vec{l}}))$ has
\begin{itemize}
	\item a pole at $s=1$ with residue equal to
	$$ \Ht(wn(T)) \log^n \Ht(wn(T)) e_{\vec{l}}(w), $$
	which is bounded as, using Corollary \ref{CVSV}
	$$ Q^{-2} \log^n Q \cdot E^{-\frac{1}{2} \sigma(\vec{l})}. $$
	\item a pole at $s=0$.
	\item a pole at $s=-1$.
\end{itemize}
	We can thus write for $0 < \epsilon < 1$, using \cite[(6.1) \& (6.2)]{Wu14} and Proposition \ref{GlobZetaFR}
\begin{align*}
	\zeta(h,1,n(T).\eis^{\reg,(n)}(\frac{1}{2},e_{\vec{l}}) &= \int_{\Re s = \epsilon} \zeta(\frac{1}{2}+s, 1, n(T).\eis^{\reg,(n)}(\frac{1}{2},e_{\vec{l}})) \Mellin{h}(-s) \frac{ds}{2 \pi i} \\
	&\quad + O_{\F, h_0, \epsilon}(Q^{\kappa-1+\epsilon} E^{-\frac{1}{2} \sigma(\vec{l})}). 
\end{align*}
	To bound the integral on the vertical line $\Re s = \epsilon$, we have by Proposition \ref{EulerZetaEis}
\begin{align*}
	&\quad \zeta(\frac{1}{2}+s,1,n(T).\eis^{\reg,(n)}(\frac{1}{2},e_{\vec{l}})) \\
	&= \frac{\partial^n}{\partial s_0^n} \mid_{s_0=\frac{1}{2}} \left\{ \frac{\zeta_{\F}(\frac{1}{2}+s+s_0)\zeta_{\F}(\frac{1}{2}+s-s_0)}{\zeta_{\F}(1+2s_0)} \cdot \prod_{v \mid \infty} \int_{\F_v^{\times}} W_{0,v}(s_0,a(y_v)) \psi_v(y_vT_v) \norm[y_v]_v^{s} d^{\times}y_v \cdot \right. \\
	&\quad \prod_{\vp<\infty, T_{\vp} \neq 0} \frac{\zeta_{\vp}(1+2s_0)}{\zeta_{\vp}(\frac{1}{2}+s+s_0)\zeta_{\vp}(\frac{1}{2}+s-s_0)} \int_{\F_{\vp}^{\times}} W_{0,\vp}(s_0,a(y_{\vp})) \psi_{\vp}(y_{\vp}T_{\vp}) \norm[y_{\vp}]_{\vp}^{s} d^{\times}y_{\vp} \cdot \\
	&\quad \left. \prod_{\vp \in \{\vp_1,\vp_2,\vp_1',\vp_2'\}} \Cond(\psi_{\vp})^{s-s_0} \zeta_{\vp,l_{\vp}}(s,s_0) \right\},
\end{align*}
	where $\zeta_{\vp,l_{\vp}}(s,s_0)$ is defined in Lemma \ref{LocZetaRatio}. From Corollary \ref{LocEstNAEx} we deduce
	$$ \extnorm{ \frac{\partial^k}{\partial s_0^k} \mid_{s_0=\frac{1}{2}} \prod_{\vp \in \{\vp_1,\vp_2,\vp_1',\vp_2'\}} \Cond(\psi_{\vp})^{s-s_0} \zeta_{\vp,l_{\vp}}(s,s_0) } \ll_{\F,\epsilon} E^{-\frac{1}{2} \Norm[\vec{l}]+\epsilon}, \quad k \leq n \leq 2. $$
	At $\vp<\infty, T_{\vp} \neq 0$, assuming $T_{\vp} = \varpi_{\vp}^{-k_{\vp}}$ we can calculate explicitly (using for example \cite[Lemma 4.7]{Wu14})
\begin{align*}
	&\quad \frac{\zeta_{\vp}(1+2s_0)}{\zeta_{\vp}(\frac{1}{2}+s+s_0)\zeta_{\vp}(\frac{1}{2}+s-s_0)} \int_{\F_{\vp}^{\times}} W_{0,\vp}(s_0,a(y_{\vp})) \psi_{\vp}(y_{\vp}T_{\vp}) \norm[y_{\vp}]_{\vp}^{s} d^{\times}y_{\vp} \\
	&= q_{\vp}^{-k_{\vp}(\frac{1}{2}+s-s_0)} \frac{1-q_{\vp}^{-\frac{1}{2}-s-s_0}}{1-q_{\vp}^{-2s_0}} - q_{\vp}^{-k_{\vp}(\frac{1}{2}+s+s_0)-2s_0} \frac{1-q_{\vp}^{-\frac{1}{2}-s+s_0}}{1-q_{\vp}^{-2s_0}} \\
	&\quad - q_{\vp}^{-(k_{\vp}-1)(\frac{1}{2}+s-s_0)} \frac{(1-q_{\vp}^{-\frac{1}{2}-s-s_0})(1-q_{\vp}^{-\frac{1}{2}-s+s_0})}{q_{\vp}-1} \frac{1-q_{\vp}^{-2k_{\vp}s_0}}{1-q_{\vp}^{-2s_0}}.
\end{align*}
	Thus we obtain the following bound
\begin{align*}
	&\quad \extnorm{ \frac{\partial^k}{\partial s_0^k} \mid_{s_0=\frac{1}{2}} \frac{\zeta_{\vp}(1+2s_0)}{\zeta_{\vp}(\frac{1}{2}+s+s_0)\zeta_{\vp}(\frac{1}{2}+s-s_0)} \int_{\F_{\vp}^{\times}} W_{0,\vp}(s_0,a(y_{\vp})) \psi_{\vp}(y_{\vp}T_{\vp}) \norm[y_{\vp}]_{\vp}^{s} d^{\times}y_{\vp} } \\
	& \ll q_{\vp}^{-k_{\vp} \epsilon} (k_{\vp} \log q_{\vp})^k \ll_k \epsilon^{-k}.
\end{align*}
	At $v \mid \infty$, we can trivially bound
	$$ \extnorm{ \frac{\partial^k}{\partial s_0^k} \mid_{s_0=\frac{1}{2}} \int_{\F_v^{\times}} W_{0,v}(s_0,a(y_v)) \psi_v(y_vT_v) \norm[y_v]_v^{s} d^{\times}y_v } \leq \int_{\F_v^{\times}} \extnorm{\frac{\partial^k}{\partial s_0^k} \mid_{s_0=\frac{1}{2}} W_{0,v}(s_0,a(y_v))} \norm[y_v]_v^{\epsilon} d^{\times}y_v \ll_{\epsilon} 1 $$
	using classical asymptotic estimation for Whittaker functions (or \cite[Proposition 4.1]{J04}). Together with convex bounds for $\zeta_{\F}$ we see for $n=0,1,2$
	$$ \extnorm{\zeta(\frac{1}{2}+s,1,n(T).\eis^{\reg,(n)}(\frac{1}{2},e_{\vec{l}}))} \ll_{\F,\epsilon} (1+\norm[s])^{\frac{[\F:\Q]}{2}+\epsilon} E^{-\frac{1}{2} \Norm[\vec{l}]+\epsilon}. $$
	We get the desired bound using \cite[(6.1) \& (6.2)]{Wu14} again.
\end{proof}
\begin{corollary}
	For a typical pattern, we have for any $\epsilon > 0$
	$$ \norm[\zeta(h,1, n(T).\Reis_0(\vec{\vp}))] \ll_{\F,h_0,\epsilon} (EQ)^{\epsilon}(Q^{\kappa-1}E^2 + E^{-2}). $$
\label{ReisZetaEst}
\end{corollary}
\begin{proof}
	This follows from (\ref{SReisSpecDecomp}) and the following bounds resulting from Lemma \ref{BaseCNA} (2), Corollary \ref{IntwRCal} and Remark \ref{IntwRCalA}:
	$$ \mu'(\vec{l}; 0) \ll_{\epsilon} E^{\epsilon}, \quad c(\vec{n}(\vec{\vp}),\vec{l}; 0) \ll E^{-\frac{1}{2}(4-\Norm[\vec{l}])} = E^{-2 + \frac{1}{2} \Norm[\vec{l}]}. $$
\end{proof}

	\subsection{Estimation for Regularized $\intL^4$-Norms}
	
	Recall $E > 0$ defined in Lemma \ref{TruncEst}. Choose a uniformizer $\varpi_{\vp}$ at every finite place $\vp$. Let $\vec{n} = (n_{\vp})_{\vp}, n_{\vp} \in \ag{N}$ such that
\begin{itemize}
	\item $n_{\vp} = 0$ unless $E \leq q_{\vp} < 2E$;
	\item both the number of $\vp$ such that $n_{\vp} \neq 0$ and $\norm[n_{\vp}]$ are bounded by some absolute constant.
\end{itemize}
	Associated to $\vec{n}$ we define $t = t(\vec{n}) \in \ag{A}^{\times}$ such that $t_v = 1, v \mid \infty$ and $t_{\vp} = \varpi_{\vp}^{-n_{\vp}}$. Take $\varphi_j=\eis^*(0,f_j)$ for $j=1,2$ where
\begin{itemize}
	\item $f_{j,v} \in \pi_v(1,1)$ is a unitary vector, spherical at each $v=\vp< \infty$;
	\item at each $v \mid \infty$, $f_{j,v}$ lies in some $\gp{K}_v$-isotypic $H_{n_{j,v}}$, whose definition is recalled in Remark \ref{IntwRCalA};
	\item write $\vec{n}_j = (n_{j,v})_{v \mid \infty}$ or $(n_{j,v})_v$ with $n_{j,\vp} = 0$ for $\vp < \infty$;
	\item $\norm[n_{j,v}], v \mid \infty$ are bounded by some absolute constant.
\end{itemize}
	Recall \cite[Definition 5.26]{Wu9} and write for $j=1,2$
	$$ \Reis_t = \Reis(a(t)\varphi_1 \cdot \overline{\varphi_2}), \quad \varphi_t = a(t)\varphi_1 \cdot \overline{\varphi_2} - \Reis_t, \quad \Reis_j = \Reis(\norm[\varphi_j]^2). $$
	We are interested in the $\intL^2$-norm of $\varphi_t$ in terms of $E$.
\begin{theorem}
	With the above notations and conditions, we have
	$$ \Norm[\varphi_t] \ll (\log E)^3. $$
\label{L4Bd}
\end{theorem}
\begin{proof}
	Note that
\begin{align}
	\Norm[\varphi_t]^2 &= \int_{[\PGL_2]}^{\reg} \left( a(t)\norm[\varphi_1]^2 - a(t)\Reis_1 \right)(g) \left( \norm[\varphi_2]^2 - \Reis_2 \right)(g) dg + \int_{[\PGL_2]}^{\reg} \norm[\varphi_2]^2(g)a(t)\Reis_1(g) dg \label{L4NrG} \\
	&\quad - \int_{[\PGL_2]}^{\reg} a(t)\Reis_1(g)\Reis_2(g) dg + \int_{[\PGL_2]}^{\reg} \extnorm{\Reis_t(g)}^2 dg \nonumber \\
	&\quad - 2 \Re \int_{[\PGL_2]}^{\reg} a(t)\varphi_1(g) \overline{\varphi_2(g)} \overline{\Reis_t(g)} dg + \int_{[\PGL_2]}^{\reg} a(t)\norm[\varphi_1]^2(g)\Reis_2(g) dg. \nonumber
\end{align}
	We bound the right hand side term by term, which will occupy the rest of this subsection. Note that only the fourth and sixth terms have growing contribution as $(\log E)^3$ and $(\log E)^6$.
\end{proof}

	For the first term in (\ref{L4NrG}), we use Cauchy-Schwarz inequality to get
\begin{align*}
	&\quad \int_{[\PGL_2]}^{\reg} \left( a(t)\norm[\varphi_1]^2 - a(t)\Reis_1 \right)(g) \left( \norm[\varphi_2]^2 - \Reis_2 \right)(g) dg \\
	&= \int_{[\PGL_2]} a(t)\left( \norm[\varphi_1]^2 - \Reis_1 \right)(g) \left( \norm[\varphi_2]^2 - \Reis_2 \right)(g) dg \\
	&\leq \extNorm{a(t)\left( \norm[\varphi_1]^2 - \Reis_1 \right)} \cdot \extNorm{\norm[\varphi_2]^2 - \Reis_2} = \extNorm{\norm[\varphi_1]^2 - \Reis_1} \cdot \extNorm{\norm[\varphi_2]^2 - \Reis_2},
\end{align*}
	which is independent of $t$ hence $E$.
	
	For the second \& third term in (\ref{L4NrG}), first notice that we can write
	$$ \Reis_1 = \sum_{k=0}^2 \Reis_{1,k}^{(k)}(\frac{1}{2}) $$
where $\Reis_{1,k}(s)$ is a regularizing Eisenstein series (\ref{RegEisDef}) and the superscript $(k)$ means taking derivative $k$ times with respect to $s$. Moreover, $\Reis_{1,k}(s)$ are spherical at finite places. We notice further that although the regularized integral is not $\GL_2(\ag{A})$-invariant, it is still $\gp{K}$-invariant \cite[Proposition 5.27 (2)]{Wu9}. Hence if we write the Hecke operator associated with $\vec{n}$ as
	$$ T(\vec{n}) := \int_{\gp{K}_{\fin} \times \gp{K}_{\fin}} \rpR(\kappa_1 a(t) \kappa_2) d\kappa_1 d\kappa_2 = \sideset{}{_{\vp < \infty}} \prod T(\vp^{\norm[n_{\vp}]}), $$
where $\rpR(\cdot)$ denotes the $\GL_2(\ag{A})$-translation, we have
	$$ \int_{[\PGL_2]}^{\reg} \norm[\varphi_2]^2(g)a(t)\Reis_1(g) dg = \int_{[\PGL_2]}^{\reg} \norm[\varphi_2]^2(g) \left( T(\vec{n})\Reis_1(g) \right) dg, $$
	$$ \int_{[\PGL_2]}^{\reg} a(t)\Reis_1(g)\Reis_2(g) dg = \int_{[\PGL_2]}^{\reg} \left( T(\vec{n})\Reis_1(g) \right) \Reis_2(g) dg. $$
	$\Reis_{1,k}(s)$ is a generalized eigenvector of $T(\vec{n})$ with eigenvalue \cite[Theorem 4.6.6]{Bu98}
\begin{equation}
	\lambda(\vec{n};s) = \sideset{}{_{\vp}} \prod \lambda_{\vp}(\norm[n_{\vp}];s), \quad \lambda_{\vp}(n;s) = \frac{q_{\vp}^{-n/2}}{1+q_{\vp}^{-1}} \left( \frac{q_{\vp}^{(n+1)s} - q_{\vp}^{-(n+1)s}}{q_{\vp}^s - q_{\vp}^{-s}} - q_{\vp}^{-1} \frac{q_{\vp}^{(n-1)s} - q_{\vp}^{-(n-1)s}}{q_{\vp}^s - q_{\vp}^{-s}} \right)
\label{HeckeEV}
\end{equation}
	in the sense that (c.f. \cite[Remark 5.17]{Wu9})
	$$ T(\vec{n})\Reis_{1,k}(1/2+s) = \lambda(\vec{n};1/2+s)\Reis_{1,k}(1/2+s) + c_k (\lambda(\vec{n};1/2+s)-1) \lambda_{\F}(s), $$
where $\lambda_{\F}(s)$ is a ratio of zeta functions of $\F$ (see for example Theorem \ref{RIPEisUnitary}) and $c_k \in \ag{C}$ is some constant depending on $\varphi_1$. Note that $\lambda(\vec{n};1/2)=1$. Consequently, we get
\begin{align*}
	T(\vec{n})\Reis_{1,k}^{(k)}(1/2) &= \sum_{l=0}^k \binom{k}{l} \lambda^{(l)}(\vec{n};1/2) \Reis_{1,k}^{(k-l)}(1/2) \\
	&\quad + c_k \sum_{l=0}^{k-1} \binom{k}{l+1} \lambda^{(l+1)}(\vec{n};1/2) \lambda_{\F}^{(k-l)}(0) + \frac{\lambda^{(k+1)}(\vec{n};1/2) \lambda_{\F}^{(-1)}(0)}{k+1}.
\end{align*}
	Inserting the obvious bound $\lambda^{(l)}(\vec{n};1/2) \ll (\log E)^l$, we see that the second and third term in (\ref{L4NrG}) are bounded as $\ll (\log E)^3$.
\begin{remark}
	$\lambda_{\vp}(n;s) = c_{\vp}(n,0;s)$ defined in Lemma \ref{BaseCNA} (2).
\end{remark}

	For the fourth term in (\ref{L4NrG}), we first make $\Reis_t$ explicit.
\begin{definition}
	We extend the definition of $\vec{l} \preceq \vec{n}$ in Definition \ref{SpNot} into the case $n_v \neq 0, v \mid \infty$ by
\begin{itemize}
	\item $\F_v = \ag{R}$: $l_v \preceq n_v$ means $l_v = n_v$;
	\item $\F_v = \ag{C}$: $l_v \preceq n_v$ means $0 \leq l_v \leq n_v$.
\end{itemize}
	At $v \mid \infty$, we write $e_{n_v}$ for some unitary vector in $H_{n_v}$ recalled in Remark \ref{IntwRCalA} without specification.
\end{definition}
\noindent We can thus write
\begin{equation}
	a(t).\varphi_1 = \sideset{}{_{\vec{l} \preceq \vec{n}} } \sum c(\vec{n},\vec{l};0) \eis^*(0,e_{\vec{n}_1+\vec{l}}) \quad \Rightarrow \quad \Reis_t = \sideset{}{_{\vec{l} \preceq \vec{n}} } \sum c(\vec{n},\vec{l};0) \Reis( \eis^*(0,e_{\vec{n}_1+\vec{l}}) \overline{\varphi_2} ).
\label{TransD}
\end{equation}
	Recall the tensor product formulas for representations of $\gp{K}_v, v \mid \infty$
\begin{itemize}
	\item $\F_v = \ag{R}$: $H_n \otimes H_m \simeq H_{n+m}$;
	\item $\F_v = \ag{C}$: $H_n \otimes H_m \simeq H_{n+m} \oplus H_{n+m-2} \oplus \cdots \oplus H_{\norm[n-m]}$.
\end{itemize}
	Together with the explicit determination of constant terms
	$$ \eisCst^*(0,e_{\vec{n}_1+\vec{l}})(a(y)\kappa) = \left\{ \Lambda_{\F}^*(1) \norm[y]_{\A}^{\frac{1}{2}} \log \norm[y]_{\A} + \left( 2\gamma_{\F} - \frac{1}{2} \Lambda_{\F}^*(1) \mu'(\vec{n}_1+\vec{l}; 0) \right) \norm[y]_{\A}^{\frac{1}{2}} \right\} e_{\vec{n}_1+\vec{l}}(\kappa), $$
	$$ \varphi_{2,\grN}(a(y)\kappa) = \left\{ \Lambda_{\F}^*(1) \norm[y]_{\A}^{\frac{1}{2}} \log \norm[y]_{\A} + \left( 2\gamma_{\F} - \frac{1}{2} \Lambda_{\F}^*(1) \mu'(\vec{n}_2; 0) \right) \norm[y]_{\A}^{\frac{1}{2}} \right\} f_2(\kappa), $$
	we deduce the existence of $c_{\vec{m}}$ such that
\begin{equation}
	\Reis( \eis^*(0,e_{\vec{n}_1+\vec{l}}) \overline{\varphi_2} ) = \sum_{\vec{m} \preceq \vec{n}_1 + \vec{n}_2} c_{\vec{m}} \left( c_2 \eis^{\reg,(2)}(\frac{1}{2},e_{\vec{m}+\vec{l}}) + c_1(\vec{l}) \eis^{\reg,(1)}(\frac{1}{2},e_{\vec{m}+\vec{l}}) + c_0(\vec{l}) \eis^{\reg}(\frac{1}{2},e_{\vec{m}+\vec{l}}) \right),
\label{TransD2}
\end{equation}
	where we have written
	$$ c_2 = c_2(\vec{l}) = \norm[\Lambda_{\F}^*]^2, \quad c_0(\vec{l}) = \left( 2\gamma_{\F} - \frac{1}{2} \Lambda_{\F}^* \mu'(\vec{n}_1+\vec{l}; 0) \right) \overline{\left( 2\gamma_{\F} - \frac{1}{2} \Lambda_{\F}^* \mu'(\vec{n}_2; 0) \right)}, $$
	$$ c_1(\vec{l}) = \left( 2\gamma_{\F} - \frac{1}{2} \Lambda_{\F}^* \mu'(\vec{n}_1+\vec{l}; 0) \right) \overline{\Lambda_{\F}^*} + \Lambda_{\F}^* \overline{\left( 2\gamma_{\F} - \frac{1}{2} \Lambda_{\F}^* \mu'(\vec{n}_2; 0) \right)}. $$
\begin{lemma}
	For $k_1,k_2 \geq 0$, the regularized integral
	$$ \int_{[\PGL_2]}^{\reg} \eis^{\reg,(k_1)}(\frac{1}{2},e_{\vec{m}_1+\vec{l}_1})(g) \overline{\eis^{\reg,(k_2)}(\frac{1}{2},e_{\vec{m}_2+\vec{l}_2})(g)} dg $$
	is non-vanishing only if $\vec{m}_1 = \vec{m}_2, \vec{l}_1 = \vec{l}_2$. We also have
	$$ \extnorm{\int_{[\PGL_2]}^{\reg} \eis^{\reg,(k_1)}(\frac{1}{2},e_{\vec{m}+\vec{l}})(g) \overline{\eis^{\reg,(k_2)}(\frac{1}{2},e_{\vec{m}+\vec{l}})(g)} dg} \ll E^{-\Norm[\vec{l}]} (\log E)^{k_1+k_2+2}. $$
\end{lemma}
\begin{proof}
	This is a direct consequence of Theorem \ref{RIPEisSing}, together with
	$$ \widetilde{\Intw}_{1/2}^{(k)}e_{\vec{m}+\vec{l}} = \left. \frac{d^k}{ds^k} \right|_{s=1/2} \left( \lambda_{\F}(s-1/2)\mu(\vec{m}+\vec{l};s) \right) e_{\vec{m}+\vec{l}}. $$
\end{proof}
\noindent Applying the lemma and the estmations
	$$ \mu'(\vec{n}_1+\vec{l}; 0)= \mu'(\vec{n_1}; 0)+\mu'(\vec{l}; 0), \quad \mu'(\vec{l}; 0) \ll \log E, \quad \norm[c(\vec{n},\vec{l};0)] \ll E^{-\frac{\Norm[\vec{n}] - \Norm[\vec{l}]}{2}} $$
we finally get
\begin{align*}
	\int_{[\PGL_2]}^{\reg} \extnorm{\Reis_t(g)}^2 dg &= \sum_{\vec{l} \preceq \vec{n}} \sum_{\vec{m} \preceq \vec{n}_1 + \vec{n}_2} \norm[c(\vec{n},\vec{l};0)c_{\vec{m}}]^2 \sum_{k_1,k_2=0}^2 c_{k_1}(\vec{l}) \overline{c_{k_2}(\vec{l})} \cdot \\
	&\quad \int_{[\PGL_2]}^{\reg} \eis^{\reg,(k_1)}(\frac{1}{2},e_{\vec{m}+\vec{l}})(g) \overline{\eis^{\reg,(k_2)}(\frac{1}{2},e_{\vec{m}+\vec{l}})(g)} dg \\
	&\ll E^{-\Norm[\vec{n}]} (\log E)^6.
\end{align*}

	For the fifth term in (\ref{L4NrG}), note that we have computed/decomposed $a(t).\varphi_1$ and $\Reis_t$ in (\ref{TransD}) and (\ref{TransD2}).
\begin{lemma}
	For any $k \geq 0$, the regularized integral
	$$ \int_{[\PGL_2]}^{\reg} \eis^*(0,e_{\vec{n}_1+\vec{l}_1})(g) \overline{\varphi_2(g)} \overline{\eis^{\reg,(k)}(\frac{1}{2},e_{\vec{m}+\vec{l}_2})(g)} dg $$
	is non-vanishing only if $\vec{l}_1 = \vec{l}_2$. We also have
	$$ \extnorm{ \int_{[\PGL_2]}^{\reg} \eis^*(0,e_{\vec{n}_1+\vec{l}})(g) \overline{\varphi_2(g)} \overline{\eis^{\reg,(k)}(\frac{1}{2},e_{\vec{m}+\vec{l}})(g)} dg } \ll_k E^{-\Norm[\vec{l}]} (\log E)^{k+4}. $$
\end{lemma}
\begin{proof}
	We apply Theorem \ref{RegTripEis}. The ``degenerate part'', i.e., the weighted sum involving $\mathrm{P}_{\gp{K}}$ is easily seen to be non-vanishing only if $\vec{l}_1 = \vec{l}_2$, and in the case $\vec{l}_1 = \vec{l}_2 = \vec{l}$, it is bounded by $E^{-\Norm[\vec{l}]} (\log E)^{k+4}$ since $\mu^{(n)}(\vec{l};1/2) \ll E^{-\Norm[l]} (\log E)^n$. For the ``main part'', we note that
	$$ R(\frac{1}{2}+\bar{s}, \eis^*(0,e_{\vec{n}_1+\vec{l}_1})(g) \overline{\varphi_2(g)}; \overline{e_{\vec{m}+\vec{l}_2}}) $$
	is the product of $\overline{\zeta_{\F}(s+1)^4 \zeta_{\F}(2s+2)^{-1}}$, some local components at $v \mid \infty$ irrelevant for estimation, and
	$$ \overline{\zeta_{\vp}( \begin{array}{cc|c} l_{1,\vp} & 0 & l_{2,\vp} \\ \overline{0} & 0 & 1/2+s \end{array} )} $$
	defined in (\ref{RSLocZetaRatioDef}). Lemma \ref{LocNVNAL4_1} (2) implies the non-vanishing assertion. We need to estimate
	$$ \left( \frac{\partial^k R}{\partial s^k} \right)^{\hol} \left( \frac{1}{2}, \cdots \right) = \frac{k!}{(k+4)!} \cdot \left. \frac{\partial^{k+4}}{\partial s^{k+4}} \right|_{s=0} \left( s^4 R(\frac{1}{2}+s, \cdots) \right). $$
	Corollary \ref{LocEstNAL4_1} concludes the bound.
\end{proof}
\noindent We deduce that
\begin{align*}
	\int_{[\PGL_2]}^{\reg} a(t)\varphi_1(g) \overline{\varphi_2(g)} \overline{\Reis_t(g)} dg &= \sum_{\vec{m} \preceq \vec{n}_1 \cdot \vec{n}_2} \sum_{\vec{l} \preceq \vec{n}} c_{\vec{m}} \norm[c(\vec{n},\vec{l};0)]^2 \cdot \sum_{k=0}^2 c_k(\vec{l}) \\
	&\quad  \int_{[\PGL_2]}^{\reg} \eis^*(0,e_{\vec{n}_1+\vec{l}})(g) \overline{\varphi_2(g)} \overline{\eis^{\reg,(k)}(\frac{1}{2},e_{\vec{m}+\vec{l}})(g)}dg \\
	&\ll E^{-\Norm[\vec{n}]} (\log E)^6.
\end{align*}

	For the last term in (\ref{L4NrG}), first notice that we can write
	$$ \Reis_2 = \sum_{k=0}^2 \eis^{\reg,(k)}(\frac{1}{2}, h_k), $$
where $h_k$ are spherical at $\vp$ for which $n_{\vp} \neq 0$. We are reduced to bounding
	$$ \int_{[\PGL_2]}^{\reg} \norm[a(t).\eis^*(0,f_1)(g)]^2 \eis^{\reg,(k)}(\frac{1}{2}, h_k)(g) dg, $$
to which apply directly Theorem \ref{RegTripEis}. The degenerate part contributes $(\log E)^3$. In fact, only
	$$ \mathrm{P}_{\gp{K}}(\Intw_0^{(l)} a(t)f_1 \cdot \overline{ a(t)f_1 } ) \mathrm{P}_{\gp{K}}(h_k), 0 \leq l \leq 3 $$
have non constant contribution. For the main part, using the $\GL_2(\F_{\vp})$-invariance of the local Rankin-Selberg zeta functions we easily deduce
	$$ R(\frac{1}{2}+s, \extnorm{ a(t).\eis^*(0,f_1) }^2; h_k) = \frac{\zeta_{\F}(s+1)^4}{\zeta_{\F}(2s+2)} \cdot \lambda(\vec{n}; \frac{1}{2}+s) \cdot \frac{\zeta_{\F}(2s+2) R(\frac{1}{2}+s, \extnorm{ \eis^*(0,f_1) }^2; h_k)}{\zeta_{\F}(s+1)^4}, $$
where only the term $\lambda(\vec{n}; \frac{1}{2}+s)$ defined in (\ref{HeckeEV}) depends on $E$. We bound
	$$ \frac{k!}{(k+4)!} \left. \frac{\partial^{k+4}}{\partial s^{k+4}} \right|_{s=0} s^4 R(\frac{1}{2}+s, \extnorm{ a(t).\eis^*(0,f_1) }^2; h_k) \ll (\log E)^{k+4} $$
from $\lambda^{(l)}(\vec{n}; 1/2) \ll (\log E)^l$, and conclude that the last term in (\ref{L4NrG}) is bounded as $\ll (\log E)^6$.

\section{Appendix: Dis-adelization in a Special Case}

	For the convenience of the readers not familiar with adelic language, and also for those who want to see the computation in the classical setting, we offer some essential part of the dis-adelized computation in the case $\F=\Q$ and $\chi = \norm_{\A}^{i\mu}$, i.e., the case for the Riemann-zeta function.
	
	On the adelic side, we shall work in the framework of $\PGL_2$, hence all previous groups such as $\gp{N}, \gp{B}, \gp{K}$ are considered as subgroups of $\PGL_2$ via the canonical projection. $1_{\fin}$ denotes the identity element in $\PGL_2(\A_{\fin})$. $\varphi$ is reserved for functions on $[\PGL_2]$. We restrict ourselves to $\SO_2(\R)$-invariant functions. The finite places of $\Q$ correspond to prime numbers $\vp = p \Z$.
	
	On the classical side, we shall confuse $\SO_2(\R)$-invariant modular forms on $\widetilde{\Gamma} \backslash \PGL_2(\R)$ with functions on $\Gamma \backslash \pH$, where $\Gamma < \PSL_2(\ag{R})$ is a lattice, $\widetilde{\Gamma}$ is the subgroup generated by $\Gamma$ and ${\rm diag}(1,-1)$. $f$ is reserved for functions on $\widetilde{\Gamma} \backslash \PGL_2(\R)$. Hence
	$$ f(z) := f\left( \begin{pmatrix} y & x \\ & 1 \end{pmatrix} \right), \quad z = x+iy \in \pH. $$
	We also write $e(x) := e^{2\pi i x}$ for $x \in \R$.
	
	We assume the existence of a compact subgroup $\gp{K}_{\Gamma} < \PGL_2(\A_{\fin})$ such that
	$$ \gp{K}_{\Gamma} \cap \PGL_2(\Q) = \widetilde{\Gamma}, \quad \det (\gp{K}_{\Gamma}) = \widehat{\Z}^{\times}, \quad a(\widehat{\Z}^{\times}) \subset \gp{K}_{\Gamma}. $$
	Then the strong approximation theorem, $\PGL_2(\A_{\fin}) = \PGL_2(\Q) \gp{K}_{\Gamma}$, yields
	$$ \PGL_2(\Q) \backslash \PGL_2(\A) / \PSO_2(\R) \gp{K}_{\Gamma} \simeq \widetilde{\Gamma} \backslash \PGL_2(\R) / \PSO_2(\R) \simeq \Gamma \backslash \pH. $$
	$\gp{K}_{\Gamma}$-invariant functions $\varphi$ are in bijective correspondence with functions $f$, determining each other by
	$$ f(g) = \varphi(g, 1_{\fin}), \quad \forall g \in \PGL_2(\R). $$
	Consequently, if $p$ denotes the standard uniformizer at the place $p$, then since $p \in \Z_q^{\times}$ for any prime $q \neq p$ and $\varphi$ is invariant by $a(\widehat{\Z}^{\times})$, we get
	$$ (a(p^{-1}).\varphi)(g,1_{\fin}) = \varphi(g,1,\cdots, 1, a(p^{-1}), 1, \cdots)) = \varphi(a(p)g,1_{\fin}) = f(a(p)g), $$
	or $(a(p^{-1})f)(z) := f(pz)$.
	The constant term, defined by
	$$ \varphi_{\gp{N}}(g) := \int_{\Q \backslash \A} \varphi(n(x)g) dx, $$
	is left-$\gp{B}(\Q)$ right-$\gp{K}_{\Gamma}$ invariant. From the strong approximation theorem, we deduce that
	$$ \gp{B}(\Q) \backslash \PGL_2(\A_{\fin}) / \gp{K}_{\Gamma} \simeq \gp{B}(\Q) \backslash \PGL_2(\Q) / \widetilde{\Gamma} $$
	corresponds bijectively with the set of cusps of $\Gamma$. (Look at the right action of $\PGL_2(\Q)$ on $\mathbb{P}^1(\Q)$!) We take a system of representatives $[\Gamma] \subset \PGL_2(\Q)$ for the above double coset decomposition. Hence $\varphi_{\gp{N}}$ and the following finite collection determine each other
	$$ \varphi_{\gp{N}}(g, \gamma), \quad \gamma \in [\Gamma] \subset \PGL_2(\Q) \subset \PGL_2(\A_{\fin}). $$
	For every $\gamma \in [\Gamma]$, $\gp{N}(\A_{\fin}) \cap \gamma \gp{K}_{\Gamma} \gamma^{-1}$ is a compact subgroup of $\gp{N}(\A_{\fin})$, hence equal to $\gp{N}(d_{\gamma} \widehat{\Z})$ for some $d_{\gamma} \in \Q^{\times}$. The strong approximation theorem for $\Q$ implies
	$$ \A_{\fin} = \Q + d_{\gamma} \widehat{\Z} \quad \Rightarrow \quad \A_{\fin} = \sideset{}{_{\alpha}} \bigsqcup \alpha + d_{\gamma}\widehat{\Z} \quad \Rightarrow \quad \Q \backslash \A = (d_{\gamma}\Z \backslash \R) \times d_{\gamma}\widehat{\Z}, $$
where $\alpha$ runs over a system of representatives for $\Q/d_{\gamma}\Z$. Hence
\begin{align*}
	\varphi_{\gp{N}}(g,\gamma) &= \int_{d_{\gamma}\Z \backslash \R} \int_{d_{\gamma}\widehat{\Z}} \varphi(n(u)g,n(z)\gamma) dz du = \norm[d_{\gamma}]^{-1} \int_{d_{\gamma}\Z \backslash \R} \varphi(n(u)g,\gamma) du \\
	&= \norm[d_{\gamma}]^{-1} \int_{d_{\gamma}\Z \backslash \R} \varphi(\gamma^{-1} n(u)g,1_{\fin}) du, \quad \forall g \in \PGL_2(\R).
\end{align*}
	On the other hand, the function
	$$ f_{\gamma}(g) := f(\gamma^{-1}g) = \varphi(\gamma^{-1}g, 1_{\fin}) $$
	is a modular form for the lattice $\gamma \widetilde{\Gamma} \gamma^{-1}$. Since
	$$ \gamma \widetilde{\Gamma} \gamma^{-1} \cap \gp{N}(\Q) = \gamma \gp{K}_{\Gamma} \gamma^{-1} \cap \gp{N}(\Q) = \gp{N}(d_{\gamma}\widehat{\Z}) \cap \gp{N}(\Q) = \gp{N}(d_{\gamma}\Z), $$
	the normalized constant term at the cusp $\infty$ of $f_{\gamma}$ is
	$$ f_{\gamma,\gp{N}}(g) := \norm[d_{\gamma}]^{-1} \int_{d_{\gamma}\Z \backslash \R} f_{\gamma}(n(x) g) dx = \varphi_{\gp{N}}(g,\gamma). $$
	Write $f_{\gp{N}} = f_{1,\gp{N}}$. Since both functions
	$$ \A^{\times} \mapsto \C, \quad y \mapsto \varphi(a(y)) \quad \text{and} \quad y \mapsto \varphi_{\gp{N}}(a(y)) $$
	are left-$\Q^{\times}$ right-$\widehat{\Z}^{\times}$ invariant, and since by class number $1$ of $\Q$
	$$ \Q^{\times} \backslash \A^{\times} / \widehat{\Z}^{\times} \simeq \{ \pm 1 \} \backslash \R^{\times} \simeq \R_{>0}, $$
	we can compute the zeta functional as
	$$ \zeta(s,1,\varphi) = \int_0^{\infty} (\varphi - \varphi_{\gp{N}})(a(t),1_{\fin}) t^{s-1/2} d^{\times}t = \int_0^{\infty} (f - f_{\gp{N}})(a(t)) t^{s-1/2} d^{\times}t. $$
	
	Write $e_0 \in H := \Ind_{\gp{B}(\A) \cap \gp{K}}^{\gp{K}} 1$ for the constant function taking value $1$ on $\gp{K}$. The normalized Eisenstein series $\eis^*(s,e_0)$ corresponds to the usual normalized real analytic Eisenstein series for the full modular group
	$$ \eis^*(s,z) := \Lambda(1+2s) \sideset{}{_{\substack{(c,d)=1 \\ c,d \in \Z}}} \sum \frac{y^{1/2+s}}{\norm[cz+d]^{1+2s}}. $$
	Introducing, first for $\Re s \gg 1$ then by analytic continuation using the second expression, the function
	$$ \BesselK_{\infty}(s,y) := \Gamma_{\R}(1+2s) \norm[y]^{1/2-s} \int_{\R} \frac{e(-uy)}{(1+u^2)^{1/2+s}} du = \norm[y]^{1/2-s} \int_0^{\infty} e^{-\pi \left( t+\frac{y^2}{t}  \right)} t^s d^{\times}t, \quad y \in \R, $$
	one computes the Fourier expansion at $\infty$ as
\begin{align*}
	(\eis^*-\eis_{\gp{N}}^*)(s,z) &= (\eis^*-\eis_{\gp{N}}^*)(s,e_0) \left( \begin{pmatrix} y & x \\ & 1 \end{pmatrix}, 1_{\fin} \right) \\
	&= \sideset{}{_{0 \neq n \in \Z}} \sum e(nx) \BesselK_{\infty}(s,ny) \cdot \norm[n]^{-1/2} \sideset{}{_{p^e \parallel n}} \prod \frac{p^{(e+1)s} - p^{-(e+1)s}}{p^s-p^{-s}} \\
	&= \sideset{}{_{0 \neq n \in \Z}} \sum e(nx) \BesselK_{\infty}(s,ny) \cdot \norm[n]^{s-1/2} \sigma_{-2s}(\norm[n]),
\end{align*}
	where $\sigma_s(m) := \sideset{}{_{d \mid m}} \sum d^s$ is the usual divisor sum function. Hence the zeta-functional is equal to
\begin{align*}
	\zeta(s,1,\eis^*(s_0,e_0)) &= \int_0^{\infty} \sideset{}{_{0 \neq n \in \Z}} \sum \BesselK_{\infty}(s_0,ny) \cdot \norm[n]^{s_0-1/2} \sigma_{-2s_0}(\norm[n]) y^{s-1/2} d^{\times}t \\
	&= 2 \sideset{}{_{n=1}^{\infty}} \sum \sigma_{-2s_0}(n) \int_0^{\infty} \int_0^{\infty} e^{-\pi \left( t+\frac{n^2y^2}{t} \right)} t^{s_0} y^{s-s_0} d^{\times}t d^{\times}y \\
	&= \sideset{}{_{n=1}^{\infty}} \sum \sigma_{-2s_0}(n) n^{s_0-s} \int_0^{\infty} \int_0^{\infty} e^{-\pi \left( t+y \right)} t^{\frac{s+s_0}{2}} y^{\frac{s-s_0}{2}} d^{\times}t d^{\times}y \\
	&= \Gamma_{\R}(s-s_0) \Gamma_{\R}(s+s_0) \zeta(s-s_0) \zeta(s+s_0) = \Lambda(s-s_0) \Lambda(s+s_0).
\end{align*}
	Proposition \ref{LocMainEst} studies a variant of the above equality by replacing $s$ with $1/2+i\mu$, $\eis^*(s_0,e_0)$ with $n(T).\eis^*(0,e_0)$ for some $T \in \R, \norm[T] \asymp \norm[\mu]$. Together with Lemma \ref{TruncEst}, we are reduced to bounding
	$$ \int_0^{\infty} h(y) \cdot (n(T).\eis^*)(0,iy) d^{\times}y, $$
	where $h$ is a positive function with support in $[\norm[\mu]^{-\kappa-1}, \norm[\mu]^{\kappa-1}]$. We would like to apply C-S
	$$ \extnorm{ \int_0^{\infty} h(y) \cdot (n(T).\eis^*)(0,iy) d^{\times}y }^2 \leq \int_0^{\infty} h(y) d^{\times}y \cdot \int_0^{\infty} h(y) \cdot (n(T).\norm[\eis^*]^2)(0,iy) d^{\times}y $$
	and Fourier inverse $\norm[\eis^*(0,z)]^2$, which is not square integrable. Michel \& Venkatesh's idea is to Fourier inverse the restriction of $\norm[\eis^*(0,z)]^2$ on a large compact region containing the domain of integration, which fails to give the Burgess-like quality because the $\intL^2$-norm of such restriction depends polynoimally on $\mu$. Our idea is to find $\Reis$, some linear combination of derivatives of Eisenstein series such that $\norm[\eis^*(0,z)]^2 - \Reis$ comes back to $\intL^2$. The existence of $\Reis$ is due to Zagier \cite{Za82} and our extension \cite[\S 5]{Wu9}. Hence we can estimate
\begin{align*}
	&\quad \int_0^{\infty} h(y) \cdot (n(T).\norm[\eis^*]^2)(0,iy) d^{\times}y \\
	&= \int_0^{\infty} h(y) \cdot (n(T).(\norm[\eis^*]^2 - \Reis) )(0,iy) d^{\times}y + \int_0^{\infty} h(y) \cdot (n(T).\Reis)(iy) d^{\times}y
\end{align*}
	the two terms on the RHS seprately. To the first term we apply Fourier inversion to $\norm[\eis^*]^2 - \Reis$. For the second, we compute directly. The advantage over Michel \& Venkatesh's approach is that the $\intL^2$-norm of $\norm[\eis^*]^2 - \Reis$ is essentially constant, while the second term contributes no more than the constant part of the the first term.
	
	The truth is a little more complicated due to the large contribution of the one-dimensional part of $\norm[\eis^*]^2 - \Reis$. Hence we need to use the method of amplification. We shall not be precise on the exact form of amplification since the goal is to give an idea how things look like. If $p$ is a prime regarded as the uniformizer of $\Q_p^{\times}$ embedded in $\A_{\fin}^{\times}$, then
	$$ \zeta(s,1,a(p^{-1}).\eis^*(s_0,e_0)) = \int_0^{\infty} (\eis^* - \eis_{\gp{N}}^*)(s_0,ipy) y^{s-1/2} d^{\times}y, $$
	from which we deduce that
	$$ \zeta(1/2+i\mu,1,a(p^{-1}).\eis^*(0,e_0)) = p^{-i\mu} \zeta(1/2+i\mu,1,\eis^*(0,e_0)). $$
	We can replace $\eis^*(0,e_0)$ with some balanced average of $p^{i\mu}a(p^{-1}).\eis^*(0,e_0)$ without affecting the integral representation and Lemma \ref{TruncEst} given above, since $a(p^{-1})$ commutes with the $n(T), T \in \R$. After applying C-S, we are lead to Fourier inversing
	$$ a(p^{-1}).\eis^*(0,e_0) \cdot \overline{\eis^*(0,e_0)} \quad \text{or} \quad \eis^*(0,pz) \cdot \overline{\eis^*(0,z)}. $$
	In the classical setting, we have regarded all forms from $\PSL_2(\Z) \backslash \pH$ to $\Gamma_0(p) \backslash \pH$. Hence we need to find $\Reis(p)$ some linear combination of derivatives of Eisenstein series for $\Gamma_0(p)$, such that
	$$ \varphi_0(p)(z) := \eis^*(0,pz) \cdot \overline{\eis^*(0,z)} - \Reis(p)(z) $$
has slow increase at every cusp of $\Gamma_0(p)$. This is of course feasible in the classical setting, but with painful computations. The more convenient adelic computation goes as follows.
\begin{definition}
	For $e \in H := \Ind_{\gp{B}(\A) \cap \gp{K}}^{\gp{K}} 1$, we write $e_s \in \Ind_{\gp{B}(\A)}^{\GL_2(\A)}(\norm_{\A}^s, \norm_{\A}^{-s})$ for the flat section associated with it. For any $s_0 \in \C, n \in \N$ we write
	$$ e_{s_0}^{(n)} := \frac{\partial^n}{\partial s^n} \mid_{s=s_0} e_s. $$
\end{definition}
\noindent It is easy to verify
	$$ e_{0,0}^{(n)} \cdot e_0^{(m)} = e_{1/2}^{(n+m)}, \quad \forall e \in H, \forall n,m \in \N. $$
	Let $e_1 \in H$ admitting the same local component as $e_0$ at places $\infty$ and $q \neq p$. At $p$ we take its local component as the $e_1$ in Lemma \ref{BCRepN}. Precisely,
	$$ e_1 \mid_{\GL_2(\Z_p) - \gp{K}_0(p)} = -p^{-1/2}, \quad e_1 \mid_{\gp{K}_0(p)} = p^{1/2}. $$
	$\eis(s,e_1)$ corresponds to an Eisenstein series for $\Gamma_0(p)$ given by
\begin{align*}
	\eis_1(s,z) &= \eis(s,e_1) \left( \begin{pmatrix} y & x \\ & 1 \end{pmatrix}, 1_{\fin} \right) = \sideset{}{_{\gamma \in \gp{B}(\Q) \backslash \PGL_2(\Q)}} \sum \sum e_{1,s}\left( \gamma \left( \begin{pmatrix} y & x \\ & 1 \end{pmatrix}, 1_{\fin} \right) \right) \\
	&= \sideset{}{_{\gamma \in \gp{N}(\Z) \backslash \widetilde{\Gamma_0(p)}}} \sum e_{1,s}\left( \gamma \left( \begin{pmatrix} y & x \\ & 1 \end{pmatrix}, 1_{\fin} \right) \right) \\
	&\quad + \sideset{}{_{\gamma \in \gp{N}_-(p\Z) \backslash \widetilde{\Gamma_0(p)}}} \sum e_{1,s}\left( \begin{pmatrix} & 1 \\ 1 & \end{pmatrix} \gamma \left( \begin{pmatrix} y & x \\ & 1 \end{pmatrix}, 1_{\fin} \right) \right)
\end{align*}
	since we know the set $\{ 0,\infty \}$ of cusps for $\widetilde{\Gamma_0(p)}$ as
	$$ \PGL_2(\Q) = \gp{B}(\Q) \widetilde{\Gamma_0(p)} \sqcup \gp{B}(\Q) \begin{pmatrix} & 1 \\ 1 & \end{pmatrix} \widetilde{\Gamma_0(p)}. $$
	We insert the value of $e_{1,s}$ to obtain
\begin{align*}
	\eis_1(s,z) &= p^{1/2} \sideset{}{_{\gamma \in \gp{N}(\Z) \backslash \widetilde{\Gamma_0(p)}}} \sum \Im(\gamma.z)^{1/2+s} - p^{-1/2} \sideset{}{_{\gamma \in \gp{N}_-(p\Z) \backslash \widetilde{\Gamma_0(p)}}} \sum \Im(\begin{pmatrix} & 1 \\ 1 & \end{pmatrix} \gamma.z)^{1/2+s}.
\end{align*}
	Note that it is regular at $s=1/2$. We now turn to the determination of $\Reis(p)$. Using the Laurent expansion of $\Lambda$ in (\ref{DedkindZExp}), we can write the constant term as
	$$ \eis_{\gp{N}}^*(0,e_0) = 2 \gamma_{\Q} e_{0,0} + \Lambda_{\Q}^* e_{0,0}^{(1)}. $$
	Computing the constant term of $a(p^{-1}).\eis^*(0,e_0)$ is convenient only if we can express $a(p^{-1})e_{0,s}$ as linear combination of flat sections. This is done in Lemma \ref{BaseCNA}, yielding
	$$ a(p^{-1}).e_{0,s} = \frac{p^{s+1/2} - p^{-(s+1/2)}}{p^{1/2} + p^{-1/2}} e_{1,s} + \frac{p^s + p^{-s}}{p^{1/2} + p^{-1/2}} e_{0,s}, $$
	$$ a(p^{-1}).e_{0,0}^{(1)} = e_{0,0}^{(1)} + \frac{p^{1/2} - p^{-1/2}}{p^{1/2} + p^{-1/2}} e_{1,0}^{(1)} + \frac{p^{1/2} - p^{-1/2}}{p^{1/2} + p^{-1/2}} (\log p) e_{0,0} + (\log p) e_{1,0}. $$
	We thus obtain
\begin{align*}
	&\quad (a(p^{-1}).\eis^*)_{\gp{N}}(0,e_0) \cdot \overline{\eis_{\gp{N}}^*(0,e_0)} \\
	&= \norm[\Lambda_{\Q}^*]^2 e_{0,1/2}^{(2)} +  \norm[\Lambda_{\Q}^*]^2 \frac{p^{1/2} - p^{-1/2}}{p^{1/2} + p^{-1/2}} e_{1,1/2}^{(2)} \\
	&+ \left( \norm[\Lambda_{\Q}^*]^2 \frac{p^{1/2} - p^{-1/2}}{p^{1/2} + p^{-1/2}} (\log p) + 2 \overline{\gamma_{\Q}} \Lambda_{\Q}^* + 2 \gamma_{\Q} \overline{\Lambda_{\Q}^*} \right) e_{0,1/2}^{(1)} \\
	&+ \left( \norm[\Lambda_{\Q}^*]^2 (\log p) + 2 \overline{\gamma_{\Q}} \Lambda_{\Q}^* \frac{p^{1/2} - p^{-1/2}}{p^{1/2} + p^{-1/2}} + 2 \gamma_{\Q} \overline{\Lambda_{\Q}^*} \frac{p^{1/2} - p^{-1/2}}{p^{1/2} + p^{-1/2}} \right) e_{1,1/2}^{(1)} \\
	&+ \left( 2 \overline{\gamma_{\Q}} \Lambda_{\Q}^* \frac{p^{1/2} - p^{-1/2}}{p^{1/2} + p^{-1/2}} (\log p) + 4 \norm[\gamma_{\Q}]^2 \right) e_{0,1/2} \\
	&+ \left( 2 \overline{\gamma_{\Q}} \Lambda_{\Q}^* (\log p) + 4 \norm[\gamma_{\Q}]^2  \frac{p^{1/2} - p^{-1/2}}{p^{1/2}} \right) e_{1,1/2}.
\end{align*}
	Writing for $\eis(s,z)$ the usual (non-completed) real analytic Eisenstein series and for $n \in \N$
	$$ \eis^{\reg,(n)}(1/2,z) := \left. \frac{\partial^n}{\partial s^n} \right|_{s=1/2} \left( \eis(s,z) - \frac{1}{s-1/2} \right), \quad \eis_1^{(n)}(1/2,z) = \left. \frac{\partial^n}{\partial s^n} \right|_{s=1/2} \eis_1(s,z), $$
	we deduce via the adelic-classical correspondence of constant terms that
\begin{align*}
	\Reis(p)(z) &:= \norm[\Lambda_{\Q}^*]^2 \eis^{\reg,(2)}(1/2,z) +  \norm[\Lambda_{\Q}^*]^2 \frac{p^{1/2} - p^{-1/2}}{p^{1/2} + p^{-1/2}} \eis_1^{(2)}(1/2,z) \\
	&+ \left( \norm[\Lambda_{\Q}^*]^2 \frac{p^{1/2} - p^{-1/2}}{p^{1/2} + p^{-1/2}} (\log p) + 2 \overline{\gamma_{\Q}} \Lambda_{\Q}^* + 2 \gamma_{\Q} \overline{\Lambda_{\Q}^*} \right) \eis^{\reg,(1)}(1/2,z) \\
	&+ \left( \norm[\Lambda_{\Q}^*]^2 (\log p) + 2 \overline{\gamma_{\Q}} \Lambda_{\Q}^* \frac{p^{1/2} - p^{-1/2}}{p^{1/2} + p^{-1/2}} + 2 \gamma_{\Q} \overline{\Lambda_{\Q}^*} \frac{p^{1/2} - p^{-1/2}}{p^{1/2} + p^{-1/2}} \right) \eis_1^{(1)}(1/2,z) \\
	&+ \left( 2 \overline{\gamma_{\Q}} \Lambda_{\Q}^* \frac{p^{1/2} - p^{-1/2}}{p^{1/2} + p^{-1/2}} (\log p) + 4 \norm[\gamma_{\Q}]^2 \right) \eis^{\reg}(1/2,z) \\
	&+ \left( 2 \overline{\gamma_{\Q}} \Lambda_{\Q}^* (\log p) + 4 \norm[\gamma_{\Q}]^2  \frac{p^{1/2} - p^{-1/2}}{p^{1/2}} \right) \eis_1(1/2,z)
\end{align*}
	does the job.
	
	With the above construction, we then showed that the (re-normalized) $\intL^2$-norm of $\varphi_0(p)$ on $\Gamma_0(p) \backslash \pH$ is $\ll (\log p)^3$, as well as its derivatives with respect to the Lie algebra of $\PGL_2(\R)$. Its Fourier inversion
\begin{align*}
	\varphi_0(p)(z) &= \frac{\Pairing{\varphi_0(p)}{1}}{\Vol(\Gamma_0(p) \backslash \pH)} \cdot 1 + \sideset{}{_{f \text{ cusp form for } \Gamma_0(p)}} \sum C_p(f) f(z) \\
	&\quad + \int_{-\infty}^{\infty} C_p(0,i\tau) \eis(i\tau,z) \frac{d\tau}{4\pi} + \int_{-\infty}^{\infty} C_p(1,i\tau) \eis_1(i\tau,z) \frac{d\tau}{4\pi} \\
	&= \varphi_0(p)_{\gp{N}}(z) + \sideset{}{_{f \text{ cusp form for } \Gamma_0(p)}} \sum C_p(f) f(z) \\
	&\quad + \int_{-\infty}^{\infty} C_p(0,i\tau) (\eis-\eis_{\gp{N}})(i\tau,z) \frac{d\tau}{4\pi} + \int_{-\infty}^{\infty} C_p(1,i\tau) (\eis_1-\eis_{1,\gp{N}})(i\tau,z) \frac{d\tau}{4\pi} 
\end{align*}
	with Fourier coefficients $C_p(f),  C_p(0,i\tau), C_p(1,i\tau) \in \C$, converges in uniformly in any compact (for the first version) resp. Siegel domain (for the second). Thus
\begin{align*}
	\zeta(h,n(T).\varphi_0(p)) &:= \int_0^{\infty} h(y) \cdot \varphi_0(p)(iy) d^{\times}y \\
	&= \int_0^{\infty} h(y) \cdot \varphi_0(p)_{\gp{N}}(iy) d^{\times}y + \sideset{}{_{f}} \sum C_p(f) \int_0^{\infty} h(y) \cdot f(iy) d^{\times}y \\
	&\quad + \int_{-\infty}^{\infty} C_p(0,i\tau) \int_0^{\infty} h(y) \cdot (\eis-\eis_{\gp{N}})(i\tau,iy) d^{\times}y \frac{d\tau}{4\pi} \\
	&\quad + \int_{-\infty}^{\infty} C_p(1,i\tau) \int_0^{\infty} (\eis_1-\eis_{1,\gp{N}})(i\tau,iy) d^{\times}y \frac{d\tau}{4\pi} 
\end{align*}
	and we estimate the RHS term by term.

\section*{Acknowledgement}

	The preparation of the paper scattered during the stays of the author's in FIM at ETHZ, YMSC at Tsinghua University, Alfr\'ed Renyi Institute in Hungary supported by the MTA R\'enyi Int\'ezet Lend\"ulet Automorphic Research Group and in TAN at EPFL. The author would like to thank all four institutes for their hospitality. The author would like to thank Professor Paul Nelson for useful discussions. The author would also like to thank Professor Philippe Michel, his former thesis adviser, for his encouragement and enlightening strategic argument. The author would like to thank Dr. Bingrong Huang for bringing the paper of Soehne \cite{Soe97} into his attention. The author is grateful for the referees' detailed check and suggestions.

\bibliographystyle{acm}
	
\bibliography{mathbib}	
	
\address{\quad \\ Han WU \\ MA C3 604 \\ EPFL SB MATHGEOM TAN \\ CH-1015, Lausanne \\ Switzerland \\ wuhan1121@yahoo.com}
	
\end{document}